\documentclass[11pt]{amsart}
\usepackage{amsmath}
\usepackage{amssymb}
\usepackage{graphicx}
\usepackage{mathrsfs}
\usepackage{tikz-cd}
\usepackage{xcolor}
\usepackage{tikz}
\usepackage{lineno}
\usepackage[utf8]{inputenc}
\usepackage[english]{babel}
\usepackage{amsthm}
\usepackage{amsbsy}

\newtheorem{thm}{Theorem}[section]
\newtheorem{prop}[thm]{Proposition}
\newtheorem{lemma}[thm]{Lemma}
\newtheorem{cor}[thm]{Corollary}

\theoremstyle{definition}
\newtheorem{defn}[thm]{Definition}
\newtheorem{rmk}[thm]{Remark}
\newtheorem{ex}[thm]{Example}

\usetikzlibrary{%
  matrix,%
  calc,%
  arrows%
}

\graphicspath{ {main/} }

\newcommand{\C}{\mathbb{C}}

\newcommand{\Z}{\mathbb{Z}}
\newcommand{\Q}{\mathbb{Q}}
\newcommand{\N}{\mathbb{N}}
\newcommand{\F}{\mathbb{F}}

\newcommand{\openp}{\pmb{(}}
\newcommand{\closep}{\pmb{)}}
\newcommand{\tensor}{\otimes}

\DeclareMathOperator{\Tr}{Tr}
\DeclareMathOperator{\Id}{id}

\DeclareMathOperator{\im}{im}

\DeclareMathOperator{\Ext}{Ext}

\DeclareMathOperator{\Hom}{Hom}
\DeclareMathOperator{\Gl}{GL}
\DeclareMathOperator{\SL}{SL}
\DeclareMathOperator{\Aut}{Aut}
\usepackage{esint}
\usepackage[backref=page, colorlinks]{hyperref}

\title[A Constructive Proof that Many Groups are Not MS]{A Constructive Proof that Many Groups with Non-Torsion 2-Cohomology are Not Matricially Stable}
\author{Forrest Glebe}
\date{July 2022}

\begin{document}
\maketitle
\begin{abstract}
A discrete group is matricially stable if every function from the group to a complex unitary group that is ``almost multiplicative'' in the point-operator norm topology is ``close'' to a genuine unitary representation. It follows from a recent result due to Dadarlat that all amenable, groups with non-torsion integral 2-cohomology are not matricially stable, but the proof does not lead to explicit examples of asymptotic representations that are not perturbable to genuine representations. The purpose of this paper is to give an explicit formula, in terms of cohomological data, for asymptotic representations that are not perturbable to genuine representations for a class of groups that contains all finitely generated groups with a non-torsion 2-cohomology class that corresponds to a central extension where the middle group is residually finite. This class includes polycyclic groups with non-torsion 2-cohomology. 
\end{abstract}

\tableofcontents

\section{Introduction}
An {\em asymptotic representation} of a discrete group $\Gamma$ is a sequence of functions $\rho_n:\Gamma\rightarrow U(k_n)$ so that for all $g,h\in\Gamma$ we have $||\rho_n(gh)-\rho_n(g)\rho_n(h)||\rightarrow0$ as $n$ goes to infinity, where $||\cdot||$ is operator norm. We say an asymptotic representation is {\em perturbable to a genuine representation} if there is a sequence of representations $\Tilde{\rho}_n:\Gamma\rightarrow U(k_n)$ so that for all $g\in\Gamma$ we have $||\rho_n(g)-\Tilde{\rho}_n(g)||\rightarrow0$ as $n$ goes to infinity. Recall that a countable discrete group, $\Gamma$, is {\em matricially stable} if every asymptotic representation of $\Gamma$ is perturbable to a genuine representation of~$\Gamma$~\cite{stab}.

In \cite{voiculescuM} Voiculescu shows that $\Z^2$ is not matricially stable by constructing an explicit sequence of pairs of unitaries that commute asymptotically in operator norm, but remain far from pairs of unitaries that commute. In~\cite{KvoiculescuM} Kazhdan independently uses the same sequence of pairs of unitaries to show that a particular surface group is not Ulam stable, where Ulam stability is defined similarly to matricial stability, but the poitwise convergence is replaced by uniform convergence. Kazhdan also connects his argument to the 2-cohomology of the group. In \cite{stab} Eilers, Shulman, and S\o rensen, give explicit asymptotic representations that are not perturbable to genuine representations for non-cyclic torsion-free finitely generated 2-step nilpotent, groups and several other groups.

In \cite{obs} Dadarlat shows that a large class of countable discrete groups with nonvanishing even rational cohomology are not matricially stable, including amenable, and hence polycyclic, groups with nonvanishing rational even cohomology. In \cite{quasireps} he connects this obstruction on the level of 2-cohomology to the ``winding number argument'' used by Kazhdan. However, the proof in \cite{obs} uses Voiculescu's theorem so it  cannot lead to an explicit construction of an asymptotic representation that is not perturbable to a genuine representation. 
In this paper we will give an alternate proof that a group with a 2-homology class that pairs nontrivially with a 2-cohomology class,~$x$ satisfying the an additional condition is not matricially stable, and give a formula, in terms of cohomological data. The following result is what we aim to prove:
\begin{thm}\label{coc2}
Suppose that $\Gamma$ is a countable discrete group and $x\in H^2(\Gamma;\Z)$ is a cohomology class represented as a central extension: 
$$\begin{tikzcd}
e\arrow{r} & \Z
\arrow{r}{\iota} &
\Tilde{\Gamma} \arrow{r}{}
& \Gamma\arrow{r}{}& e.
\end{tikzcd}$$
If $x$ is not in the kernel of the map $h:H^2(\Gamma;\Z)\rightarrow\Hom(H_2(\Gamma;\Z),\Z)$ induced by the Kronecker pairing and there is a sequence of finite index subgroups $\Gamma_n\le\tilde\Gamma$ so that $\iota(\Z)\cap\bigcap\Gamma_n=\{e\}$ then $\Gamma$ is not matricially stable. The sequence of functions $\rho_n$ we will define in Proposition~\ref{rhon} is an asymptotic representation of $\Gamma$ that cannot be perturbed to a genuine representation. In fact the asymptotic representation may not be perturbed to any representation let alone a unitary one.
\end{thm}

In particular if $\tilde\Gamma$ is residually finite  and $x$ is not in the kernel of the map $h:H^2(\Gamma;\Z)\rightarrow\Hom(H_2(\Gamma;\Z),\Z)$ then we can create an explicit formula for an asymptotic representation that cannot be perturbed to a genuine representation in terms of finite quotients of $\Gamma$ and cocycle representatives of $x$. If $\Gamma$ is finitely generated the condition that~$x$ is not in the kernel of $h$ is equivalent to the condition that~$x$ non-torsion. In particular it follows that any virtually polycyclic group with non-torsion 2-cohomology is not matricially stable and an explicit formula for the relevant asymptotic representation can be found in terms of cohomological data. 

Our construction is similar to another construction of projective representations of subgroups of $\Z^2\rtimes\SL_2(\Z)$ that come from factoring cocycles through finite quotients; see the proof of Corollary~B in~\cite{llp}.

Three virtues of our proof compared to Dadarlat's broader result are as follows. First, our proof leads to a formula for asymptotic representations that cannot be perturbed to genuine representations (Proposition~\ref{rhon}). We use this formula to construct new examples of asymptotic representations that cannot be perturbed to genuine representations in Section~5. Second, our proof is relatively elementary and uses only basic group cohomology, instead of employing techniques used in the Novikov conjecture. Third, because we do not use these techniques, we do not require the existence of a $\gamma$-element.

This paper is organized as follows. In Section~2 we will review relevant background information. In Section~3 we prove the main result (Theorem~\ref{coc2}), and find a formula for an asymptotic representation that cannot be perturbed to a genuine representation for a group satisfies the assumptions of the main theorems (Proposition~\ref{rhon}). We show that the main results apply to virtually polycyclic groups with non-torsion 2-cohomology (Corollary~\ref{polycresult}), and hence non-cyclic finitely generated torsion-free nilpotent groups. In Section~4 we give an alternate proof non-cyclic torsion-free finitely generated nilpotent groups satisfy the cohomological conditions required for the main result, which is useful for computing examples.In Section~5 we show how to recognize the cohomological conditions from generators and relations of a finitely generated group (Theorem~\ref{grmain}). In Section~6 we give an abstract motivation for our formula, showing that it can be made from an induced representation of a finite quotient of a central extension of the original group (Theorem~\ref{cext2}). In Section~7 we illustrate our formula for the following groups: $\Z^2$ to show our methods can recover Voiculescu's matrices, a 3-step nilpotent group, and the polycyclic group $\Z^2\rtimes\Z$ where the action of $\Z$ on $\Z^2$ is induced by ``Arnold's cat map.''

\section{Background Information}

\subsection{Group Homology and Cohomology}

There are many ways to characterize group homology and cohomology, but to us the most useful will be to describe them as the homology and cohomology of an explicit chain complex described below. We will only use homology with coefficients in $\Z$ and cohomology with trivial action in this paper. For more about this construction see \cite[Chapter II.3]{coho}.

\begin{defn}
Let $\Gamma$ be a discrete group. We define $C_n(\Gamma)$ to be the free abelian group generated by elements of $\Gamma^n$. We may write an element of $\Gamma^n$ as $[a_1|a_2|\cdots|a_n]$ with $a_i\in\Gamma$. We thus write a typical element of $C_n(\Gamma)$ as
$$c=\sum_{i=1}^Nx_i[a_{i1}|a_{i2}|\cdots|a_{in}]$$
with $x_i\in\Z$ and $a_{ij}\in\Gamma$. We define the boundary map $\partial_n$ from $C_n(\Gamma)\rightarrow C_{n-1}(\Gamma)$ by
$$\partial_n[a_1|\cdots|a_n]=$$
\vspace{-7mm}
$$[a_2|\cdots|a_{n}]+\sum_{i=1}^{n-1}(-1)^i[a_1|\cdots|a_{i-1}|a_ia_{i+1}|a_{i+2}|\cdots|a_n]+(-1)^n[a_1|\cdots|a_{n-1}].$$
Often we will just write $\partial$ where the domain is clear from context. The group {\em homology} of $\Gamma$ is the homology group the chain complex $(C_\bullet(\Gamma),\partial_\bullet)$, that is to say $H_n(\Gamma):=\ker(\partial_n)/\im(\partial_{n+1})$. \end{defn}
\begin{defn}
If $A$ is any abelian group then the {\em cohomology of $\Gamma$ with coefficients in $A$} is the cohomology of $(C_\bullet,\partial_\bullet)$ with coefficients\footnote{In general $A$ may be taken to be a left $\Z[\Gamma]$  module, but we will only consider the case where the action of $\Gamma$ on $A$ is trivial here, so we may consider $A$ to only have the structure of an abelian group.} in~$A$. 

To be more explicit, we use the notation
$$C^n(\Gamma;A):=\Hom(C_n(\Gamma),A)$$
and note that this is isomorphic to the group of functions from $\Gamma^n$ to $A$. Then we define $\partial^n:C^n(\Gamma;A)\rightarrow C^{n+1}(\Gamma;A)$ to be the adjoint\footnote{Actually \cite{coho} defines the coboundary map to be $(-1)^{n+1}$ times the adjoint, but this does not change the image or kernel boundary of the maps so it leads to an equivalent definition of the cohomology groups.} of $\partial_{n+1}$:
$$(\partial^n\sigma)(a_1,\ldots,a_{n+1})=$$
\vspace{-7mm}
$$\sigma(a_2,\ldots, a_{n+1})+\sum_{i=1}^n(-1)^i\sigma(a_1,\ldots, a_{i-1},a_ia_{i+1},a_{i+2},\ldots, a_{n+1})+(-1)^{n+1}\sigma(a_1,\ldots,a_n).$$
Then we define $H^n(\Gamma;A):=\ker(\partial^n)/\im(\partial^{n-1})$. As with homology we will suppress the $n$ in $\partial^n$ if the dimension is obvious from context.
\end{defn}

Suppose that $f$ is a group homomorphism from $\Gamma_1$ to $\Gamma_2$. This induces a map $f_\#$ from $C_n(\Gamma_1)$ to $C_n(\Gamma_2)$ and a map $f_*:H_n(\Gamma_1)\rightarrow H_n(\Gamma_2)$. Similarly the adjoint of $f_\#$ from $C^n(\Gamma_2)\rightarrow C^n(\Gamma_1)$ called $f^\#$ descends to a well-defined map from $f^*:H^2(\Gamma_2;A)\rightarrow H^2(\Gamma_1;A)$. Similarly if $g:A_1\rightarrow A_2$ is a homomorphism of abelian groups there is a map $g_\#$ from $C^n(\Gamma;A_1)\rightarrow C^n(\Gamma;A_2)$. This map descends to a well-defined map $g_*:H^n(\Gamma;A_1)\rightarrow H_n(\Gamma;A_2)$. All maps defined in this paragraph are functorial.

Because $C^n(\Gamma;A)$ is isomorphic to $\Hom(C_n(\Gamma),A)$ there is a natural bilinear map, called the Kronecker pairing, from $C^n(\Gamma;A)\times C_n(\Gamma)\rightarrow A$ defined by
$$\left\langle\sigma,\sum_{i=1}^Nx_i[a_{i1}|\cdots|a_{in}]\right\rangle=\sum_{i=1}^Nx_i\sigma(a_{i1},\ldots,a_{in}).$$
This descends to a well-defined bilinear map from $H^n(\Gamma;A)\times H_n(\Gamma)\rightarrow A$. We will use the notation $\langle\cdot,\cdot\rangle$ for both maps.
\subsection{2 (Co)homology}

\begin{thm}[{\cite[Theorem II.5.3]{coho}}] \label{hopf}
Suppose that $\Gamma\cong F/R$ where $F$ is a free group. Then
$$H_2(\Gamma)\cong\frac{R\cap[F,F]}{[F,R]}.$$
\end{thm}

An explicit isomorphism from the the right side of the equation to $H_2(\Gamma)$ as described above can be defined as follows. Suppose that
$$r_1=\prod_{i=1}^N[{a}_i,{b}_i]\in R\cap[F,F]$$
then by \cite[chapter II.5 Exercise 4]{coho}, 
\[r_1\mapsto \sum_{i=1}^N([I_{i-1}|\dot a_i]+[I_{i-1}\dot a_i|\dot b_i]-[I_{i-1}\dot a_i\dot b_i\dot a_i^{-1}|\dot a_i]-[I_i|\dot b_i])+\im(\partial_3)\]
where $\dot a_i$ and $\dot b_i$ are the images in $\Gamma$ of ${a}_i$ and ${b}_i$ and $I_i=[\dot a_1,\dot b_1]\cdots[\dot a_i,\dot b_i]$.

The 2-cohomology has an alternative characterization.

\begin{defn}
If $\Gamma$ is a discrete group and $A$ is an abelian group, then a {\em central extension} of $\Gamma$ by $A$ is a short exact sequence
$$\begin{tikzcd}
e\arrow{r} & A
\arrow{r}{} &
\Tilde{\Gamma} \arrow{r}{}
& \Gamma\arrow{r}{}& e
\end{tikzcd}$$
where the image of $A$ in $\tilde{\Gamma}$ is central in $\tilde{\Gamma}$. We say two central extensions are {\em equivalent} if we can make a commutative diagram as follows:
$$\begin{tikzcd}
e\arrow{r} & A
\arrow{r}{}\arrow{d}{\Id_A} &
\Tilde{\Gamma}_1\arrow{d}{\cong} \arrow{r}{}
& \Gamma\arrow{r}{}\arrow{d}{\Id_\Gamma}& e
\\
e\arrow{r} & A
\arrow{r}{} &
\Tilde{\Gamma}_2 \arrow{r}{}
& \Gamma\arrow{r}{}& e.
\end{tikzcd}$$
\end{defn}

\begin{thm}[{\cite[Theorem IV.3.12]{coho}}] As a set $H^2(\Gamma;A)$ is in bijection with the equivalence classes of central extensions of $\Gamma$ by $A$.
\end{thm}

Given an explicit central extension we may find a cocycle representative of the corresponding element of $H^2(\Gamma;A)$ as follows. Pick $\theta$ to be a set theoretic section from $\Gamma$ to $\Tilde{\Gamma}$. Then viewing $A$ as a subset of $\Tilde{\Gamma}$ define
$$\sigma(g,h)=\theta(g)\theta(h)\theta(gh)^{-1}\in A.$$
By \cite[Equation IV.3.3]{coho} this is a cocycle representative of the cohomology class corresponding to this central extension.

\subsection{Polycyclic and Nilpotent Groups}
\begin{defn}
A group $\Gamma$ is called {\em polycyclic} if there is a sequence of subgroups
$$\Gamma=\Gamma_1\ge \Gamma_2\ge\cdots\ge\Gamma_m\ge\Gamma_{m+1}=\{e\}$$
so that $\Gamma_{i}\triangleright\Gamma_{i+1}$ and $\Gamma_i/\Gamma_{i+1}$ is cyclic. A sequence of subgroups obeying this condition is called a {\em polycyclic sequence of subgroups}. We may pick this sequence so that each quotient is nontrivial. We may pick~$a_i$ to be a representative of a generator of $\Gamma_i/\Gamma_{i+1}$. We call these generators a {\em polycyclic sequence} for $\Gamma$, and they generate $\Gamma$.
\end{defn}

\begin{defn}
A group is called {\em virtually polycyclic} if it has a finite index polycyclic subgroup.
\end{defn}

In this case we may assume that there is a {\em normal} finite index polycyclic subgroup. This may be constructed by intersecting over each conjugates of the subgroup and using the fact that a subgroup of a polycyclic group is polycyclic \cite[Proposition 9.3.7]{sims}. By \cite[Theorem 3]{Hirsch} polycyclic groups are residually finite. From this it follows that virtually polycyclic groups are residually finite as well. 

\begin{prop}\label{virtualPext}
Suppose that $\Gamma$ is virtually polycyclic, and
$$\begin{tikzcd}
e\arrow{r} & \Z
\arrow{r}{\iota} &
\Tilde{\Gamma} \arrow{r}{\varphi}
& \Gamma\arrow{r}{}& e
\end{tikzcd}$$
is an extension of $\Gamma$. Then $\tilde{\Gamma}$ is virtually polycyclic as well.
\end{prop}

\begin{proof}
Let $P\subseteq\Gamma$ be a finite index normal polycyclic subgroup of $\Gamma$. Since $\iota(\Z)\subseteq\varphi^{-1}(P)$ we have that
$$\Gamma/P\cong\frac{\tilde{\Gamma}/\iota(\Z))}{\varphi^{-1}(P)/\iota(\Z))}\cong\tilde{\Gamma}/\varphi^{-1}(P).$$
Thus we may show that $\varphi^{-1}(P)$ is polycyclic. If $\{P_i\}_{i=1}^{m+1}$ is a polycyclic sequence of subgroups for $P$ then we have for $i\le m$
$$\varphi^{-1}(P_i)/\varphi^{-1}(P_{i+1})\cong\frac{\varphi^{-1}(P_i)/\iota(\Z)}{\varphi^{-1}(P_{i+1})/\iota(\Z)}\cong\Gamma_i/\Gamma_{i+1}$$
and $P_{m+1}/\{1\}\cong\Z$. Thus these make a polycyclic sequence of subgroups.
\end{proof}

By \cite[Proposition 9.3.4]{sims} all finitely generated nilpotent groups are polycyclic. Let $\Gamma$ be a torsion-free finitely generated nilpotent group.
\begin{defn}
A {\em Mal'cev basis} for a Torsion-Free finitely generated nilpotent group is an $m$-tuple of elements, $(a_1,\ldots a_m)\in\Gamma^m$ so that obeys the following conditions
\begin{itemize}
    \item For all $g\in\Gamma$, $g$ can be written uniquely as $g=a_1^{x_1}\cdots a_m^{x_m}$ for some $(x_1,\ldots, x_m)\in\Z^m$.  We call this presentation the {\em canonical form} of $g$.
    \item The subgroups $\Gamma_i=\langle a_i,\ldots, a_m\rangle$ form a central series for $\Gamma$.
\end{itemize}
\end{defn}

Every finitely generated torsion-free nilpotent group has a Mal'cev basis by \cite[Lemma 8.23]{Comp}. It also follows that $\Gamma_i$ make a polycyclic sequence of subgroups.

\subsection{Rational 2-Cohomology of a Torsion-Free Finitely Generated Nilpotent Group}

We will need the following result. 
\begin{thm}[{\cite{pickel},\cite{lie_groups_and_algebras}}]\label{Qcoho}
Let $\Gamma$ be a torsion-free finitely generated nilpotent group that is neither $\Z$ nor the trivial group. Then $H^2(\Gamma;\Q)\not\cong\{0\}$.
\end{thm}
\begin{proof}
By a result of Pickel \cite{pickel} we have that $H^\bullet(\Gamma;\Q)$ can be calculated in terms of the of the cohomology of an associated rational Lie algebra. By a result of Ado explained on \cite[pg 86]{lie_groups_and_algebras} the 2-cohomology of the algebra is nonzero.  
\end{proof}

\begin{cor}\label{nilpairing}
If $\Gamma$ is a torsion-free finitely generated nilpotent group that is neither $\Z$ nor trivial there is a pair $([\sigma],c)\in H^2(\Gamma;\Z)\times H_2(\Gamma)$ so that $\langle[\sigma],c\rangle\ne0$.
\end{cor}
\begin{proof}
The rational 2-cohomology of $\Gamma$ is nontrivial by Theorem~\ref{Qcoho}. By the universal coefficient theorem \cite[Theorem 53.1]{eat} we have a sequence
$$\begin{tikzcd}
0\arrow{r}{}&\Ext(H_1(\Gamma),\Q)\arrow{r}{}& H^2(\Gamma;\Q)\arrow{r}{}&\Hom(H_2(\Gamma),\Q)\arrow{r}{}&0
\end{tikzcd}.$$
Since $\Hom(\bullet,\Q)$ is exact $\Ext(H_1(\Gamma),\Q))\cong\{0\}$. From this it follows that $H_2(\Gamma)$ is non-torsion. Next we need to show that $H_2(\Gamma)$ is finitely generated. First we will show that $H^n(\Gamma;\Z)$ is finitely generated for all $n$ and for all torsion-free finitely generated nilpotent groups by induction on the number of elements in the Mal'cev basis. If this number is zero, this is obvious. For the inductive step let $x_1,\ldots,x_m$ be a Mal'cev basis. Then by the inductive hypothesis $H^n(\Gamma/\langle x_m\rangle;\Z)$ is finitely generated for all $n$. By \cite[Theorem~5.3]{gysin} there is an exact sequence
$$\begin{tikzcd}
\cdots H^{n+2}(\Gamma/\langle x_m\rangle;\Z)\arrow{r}{}&H^{n+2}(\Gamma;\Z)\arrow{r}{}&H^{n+1}(\Gamma/\langle x_m\rangle)\cdots
\end{tikzcd}$$
which shows that $H^{n}(\Gamma;\Z)$ is finitely generated for $n\ge2$. For $n=0$ this is obvious and for $n=1$ this follows from the fact that $H^1(\Gamma;\Z)=\Hom(H_1(\Gamma),\Z)=\Hom(\Gamma/[\Gamma,\Gamma],\Z)$~\cite[page 36]{coho}. Next the fact that the homology is finitely generated as well, follows from~\cite[Proposition 3F.12]{hatcher}. Because $H_2(\Gamma)$ is non-torsion and finitely generated we know $\Hom(H_2(\Gamma),\Z)$ is nonzero. In particular there must be a cocycle $[\sigma]$ that is not in the kernel of the map from $H^2(\Gamma;\Z)$ to $\Hom(H_2(\Gamma),\Z)$.
\end{proof}
\subsection{Log of a Matrix}

Throughout this article if $m$ is a matrix we define $\log(m)$ to be the power series for log centered at 1. That is
$$\log(m):=\sum_{j=1}^\infty(-1)^{j-1}\frac{(m-\Id)^n}{n}$$
which converges and is continuous for $||m-\Id||<1$. We will consider this to be well-defined when $||m-\Id||<1$. By looking at the Jordan form for $m$ we can see that the eigenvalues of $\log(m)$ are the logs of eigenvalues of $m$. This justifies the formula $\exp(\Tr(\log(m))=\det(m)$.

\subsection{Voiculescu's Matrices and the Winding Number Argument}

The classical example due to Voiculescu in \cite{voiculescuM} for an asymptotic representation of $\Z^2$ that is not perturbable to a genuine representation comes in the form:
$$\rho_n(a,b)=u_n^av_n^b$$
where $u_n$ and $v_n$ are $n\times n$ matrices such that
$$u_n=\begin{bmatrix}
0&0&\cdots&0&0&1\\
1&0&\cdots&0&0&0\\
0&1&\cdots&0&0&0\\
\vdots&\vdots&\ddots&\vdots&\vdots&\vdots\\
0&0&\cdots&1&0&0\\
0&0&\cdots&0&1&0
\end{bmatrix}\mbox{ and }
v_n=\begin{bmatrix}
\exp\left(\frac{2\pi i}{n}\right)&0&0&\cdots&0\\
0&\exp\left(\frac{4\pi i}{n}\right)&0&\cdots&0\\
0&0&\exp\left(\frac{6\pi i}{n}\right)&\cdots&0\\
\vdots&\vdots&\vdots&\ddots&\vdots\\
0&0&0&\cdots&1
\end{bmatrix}.
$$
The argument that we summarize here was first applied to this problem by Exel and Loring in~\cite{ELVoiculescuM}, and had previously, idependently been used by Kazhdan in ~\cite{KvoiculescuM}. It can be computed that $u_nv_nu_n^{-1}v_n^{-1}=\exp\left(\frac{-2\pi i}{n}\right)\Id_{\C^n}$. It is not difficult to show that the fact that this gets arbitrarily close to $\Id_{\C^n}$ in operator norm implies asymptotic multiplicativity of the associated representation. A sketch of the argument that this asymptotic representation cannot be perturbed follows. 

The path $p(t)=\det\left(t\exp\left(\frac{-2\pi i}{n}\right)\Id_{\C^n}+(1-t)\Id_{\C^n}\right)$ is a path in $\C^\times$ with winding number $-1$. Suppose for contradiction that $u_n$ is close enough in operator norm to a $u_n'$ and $v_n$ is close enough to $v_n'$ so that $v_n'u_n'=u_n'v_n'$ we make a contradiction as follows. We define
\begin{align*}
\hat{u}_n(s)&=su_n+(1-s)u_n'\\
\hat{v}_n(s)&=sv_n+(1-s)v_n'
\end{align*}
then
$$h(t,s)=\det(t\hat{u}_n(s)\hat{v}_n(s)\hat{u}_n(s)^{-1}\hat{v}_n(s)^{-1}+(1-t)\Id_{\C^n}).$$
It can be shown that $h(t,s)\ne0$ for all $s,t\in[0,1]$. It follows that $h$ is a homotopy from the path $p$ to the trivial loop centered at 1. This is a contradiction since~$p$ has nonzero winding number. 

A more general statement of this type of invariant can be found in \cite[Theorem 3.9]{stab} or in~\cite{quasireps}. Essentially the relation $u_nv_nu_n^{-1}v_n^{-1}$ can be replaced with another product of commutators. Note that the winding number of $p$ could also be computed by calculating $\Tr\left(\log\left(u_nv_nu_n^{-1}v_n^{-1}\right)\right)/(2\pi i)$, where $\log$ is defined to be a power-series centered at 1. We will phrase our analogous argument in terms of computing trace of log instead of computing the winding number directly, but it is inspired by this more classical argument. Kazhdan develops this example independently to show that a certain surface group is not uniformly stable \cite{KvoiculescuM}. He develops them as a representation of a central extension of the group in question, thereby connecting asymptotic representations to 2-cohomology.

\section{Main Results}

In Subsection 3.1 we prove some analytic lemmas that we use later. In subsection ~3.2 we develop a pairing between between almost multiplicative functions from $\Gamma$ to $M_n$ and 2-chains (Definition~\ref{pairing}). We show that if an almost multiplicative function is close enough to a genuine unitary representation its pairing with a 2-cycle is zero (Theorem~\ref{obstruction}). In subsection~3.3 we introduce a ``finite type'' condition on cohomology classes (Definition~\ref{finiteT}), and give some alternate characterizations of the definition (Proposition~\ref{finiteQ}). In subsection~3.4 we develop a formula for an asymptotic representation (Proposition~\ref{rhon}), and show that if the right cohomological conditions hold it is well-defined and cannot be perturbed to a genuine representation (Theorem~\ref{coc2}). We show that polycyclic groups with non-torsion 2-cohomology meet this condition (Corollary~\ref{polycresult}).

\subsection{Analytic Lemmas}
We will use the following elementary results. 
\begin{lemma}\label{anal}
Let $A$ be a $C^*$-algebra. Then the following hold:
\begin{enumerate}
    \item If $a_i,b_i\in A$ for $i\in\{1,\ldots,N\}$ so that for all $i$, $||a_i-b_i||<\varepsilon$ and $||a_i||,||b_i||<M$ we have that
    $$\left|\left|\prod_{i=1}^Na_i-\prod_{i=1}^Nb_i\right|\right|<NM^{N-1}\varepsilon.$$
    \item If $a\in A$ and $u$ is a unitary in $A$ so that $||a-u||\le\frac12$ we have
    $$||a^{-1}-u^*||\le2||a-u||.$$
    In particular $||a^{-1}||\le2$.
\end{enumerate}
\end{lemma}
\begin{proof}
(1) We compute
\begin{align*}
\sum_{j=1}^N\left(\prod_{i=1}^{j-1} a_i\right)(a_j-b_j)\prod_{i=j+1}^Nb_i&=\sum_{j=1}^N\left(\prod_{i=1}^{j}a_i\right)\prod_{i=j+1}^Nb_i-\sum_{j=1}^N\left(\prod_{i=1}^{j-1}a_i\right)\prod_{i=j}^Nb_i\\
&=\sum_{j=1}^N\left(\prod_{i=1}^{j}a_i\right)\prod_{i=j+1}^Nb_i
-\sum_{j=0}^{N-1}\left(\prod_{i=1}^{j}a_i\right)\prod_{i=j+1}^Nb_i
\\
&=\prod_{i=1}^Na_i-\prod_{i=1}^Nb_i.
\end{align*}
Applying the triangle inequality and submultiplicativity to the first term gives us the desired inequality.
\\
\\
(2) Let $b=u^*a$. Note that $||b-1||=||a-u||<1$ so by \cite[Theorem 1.2.2]{murph} we have that
$$b^{-1}=\sum_{k=0}^\infty(1-b)^k.$$
Thus
\begin{align*}
||a^{-1}-u^*||&=||a^{-1}u-1||\\
&=||b^{-1}-1||\\
&\le\sum_{k=1}^\infty||1-b||^k\\
&=\frac{||1-b||}{1-||1-b||}\\
&=\frac{||a-u||}{1-||a-u||}\\
&\le2||a-u||.
\end{align*}
We use that $||a-u||\le\frac12$ in the last step.
\end{proof}
\begin{lemma}\label{anal2}
Let $m_1,m_2\in U(n)$ so that $||m_i-\Id_{\C^n}||<\frac12$. Then if $\log$ is defined as a power series centered at 1 we have that
$$\Tr(\log(m_1m_2))=\Tr(\log(m_1))+\Tr(\log(m_2)).$$
\end{lemma}
\begin{proof}
First note that $||m_1m_2-1||<\frac12\cdot2=1$ by Lemma~\ref{anal} so the expression is well-defined. Then
$$\exp(\Tr(\log(m_1m_2))-\Tr(\log(m_1))-\Tr(\log(m_2))=\frac{\det(m_1m_2)}{\det(m_1)\det(m_2)}=1$$
so
$$\Tr(\log(m_1m_2))-\Tr(\log(m_1))-\Tr(\log(m_2))\in2\pi i\Z.$$
Then
$$\Tr(\log((m_1t+(1-t)\Id_{\C^n})m_2)-\Tr(\log(m_1t+(1-t)\Id_{\C^n})-\Tr(\log(m_2))$$
is well defined for all $t\in[0,1]$ by and in $2\pi i\Z$ by the same argument above. Because this expression depends continuously on $t$ we must have that it is constant in $t$. Plugging in $t=0$ we must have that the expression is uniquely zero.
\end{proof}

\subsection{A Homological Version of the Winding Number Argument}
The idea of this section is to find a pairing between maps from $\Gamma$ to $M_n$ that are ``almost multaplicative'' and elements of $2$-cycles in $C_2(\Gamma)$. In general how ``close'' to being multiplicative depends on the specific element of $C_2(\Gamma)$ we pair with.

\begin{defn}\label{pairing}
Suppose that $c\in C_2(\Gamma)$ is expressed by the formula
$$c=\sum_{i=1}^N x_i[a_i|b_i]$$
so that $(a_i,b_i)=(a_j,b_j)$ implies $i=j$. The {\em support} of $c$ is the set of ordered pairs $\{(a_i,b_i)\}_{i=1}^N$. The {\em boundary support of $c$} is the set of elements of $\Gamma$, $\{a_i,b_i,a_ib_i\}_{i=1}^N$. Then we say $\rho:\Gamma\rightarrow M_n(\C)$ is {\em $\varepsilon$-almost multiplicative on the support of $c$} if $\rho(a_i),\rho(b_i)\in\Gl_n(\C)$ and
$$||\rho(a_ib_i)\rho(a_i)^{-1}\rho(b_i)^{-1}-\Id_{\C^n}||<\varepsilon$$
for each $i$. In this case, if additionally $\varepsilon\le1$, we define
$$\openp\rho,c\closep=\frac{1}{2\pi i}\sum_{j=1}^Nx_j\Tr(\log(\rho(a_jb_j)\rho(b_j)^{-1}\rho(a_j)^{-1}))$$
where $\log$ is defined as a power series centered at 1.
\end{defn}

This is clearly $\Z$-linear in the second entry, in the sense that
$$\openp\rho,c_1\pm c_2\closep=\openp\rho,c_1\closep\pm\openp\rho,c_2\closep$$
when the right side is well-defined. Due to potential cancellation, the support of $c_1+c_2$ may be smaller than the support of $c_1$ union the support of~$c_2$. It is also ``linear'' in the first entry in the sense that
$$\openp\rho_1\oplus\rho_2,c\closep=\openp\rho_1,c\closep+\openp\rho_2,c\closep;$$
if one side if this equality is well-defined then so is the other because
$$||(\rho_1\oplus\rho_2)(gh)-(\rho_1\oplus\rho_2)(g)(\rho_1\oplus\rho_2)(h)||=\max_{i}||\rho_i(gh)-\rho_i(g)\rho_i(h)||.$$
\begin{prop}
If $\partial c=0$ then
$$\openp\rho,c\closep\in\Z.$$
\end{prop}
\begin{proof}
Let $F$ be the boundary support of $c$ and let $C_1(F)$ be the subgroup of $C_1(\Gamma)$ spanned by elements of the form $[g]$ where $g\in F$. Define a homomorphism $\varphi:C_1(F)\rightarrow \C^\times$ by taking $[g]\mapsto\det(\rho(g))$. This is well-defined because $C_1(F)$ is a free abelian group and $\det(\rho(g))\in\C^\times$ since $\rho(F)\subset\Gl_n(\C)$. Then we see that
$$\exp(2\pi i\openp\rho,c\closep)=\prod_{j=1}^N\det((\rho(a_jb_j)\rho(b_j)^{-1}\rho(a_j)^{-1})^{x_i})=\varphi(-\partial c)=\varphi(0)=1.$$
\end{proof}

\begin{defn}
If $c$ is a 2-cycle on $\Gamma$ with boundary support $F$ and $\rho_0$ and $\rho_1$ are maps from $\Gamma$ to $\Gl_n(\C)$ that are 1-almost multiplicative on the support of~$c$ we say that $\rho_0$ and $\rho_1$ are {\em homotopy equivalent on the boundary support of~$c$} if the following conditions are met. There is a family of functions
$$\rho^t:\Gamma\rightarrow M_n(\C)$$
continuous in $t$ so that $\rho^0(g)=\rho_0(g)$ and $\rho^1(g)=\rho_1(g)$ for all $g\in F$ and for all~$t$, $\rho^t$ is 1-almost multiplicative on the support of $c$.
\end{defn}

\begin{prop}\label{homotop}
Let $c$ be a 2-cycle on $\Gamma$ and let $0<\varepsilon<1$. Let $\rho_0:\Gamma\rightarrow \Gl_n(\C)$ be $\varepsilon$ multiplicative on the support of $c$ and $\rho_1:\Gamma\rightarrow U(n)$. 
\begin{enumerate}
    \item If $$||\rho_0(g)-\rho_1(g)||<\frac{1-\varepsilon}{24}$$ for all $g$ in the boundary support of $c$ then $\rho_0$ and $\rho_1$ are homotopy equivalent in the boundary support of $c$ and $\rho_1$ is 1-almost multiplicative on the support of $c$.
    \item If $\rho_0$ and $\rho_1$ are homotopy equivalent in the boundary support of $c$ then $\openp\rho_0,c\closep=\openp\rho_1,c\closep$.
\end{enumerate}
\end{prop}
\begin{proof}
(1) Define $\rho^t$ to be
$$t\rho_1+(1-t)\rho_0.$$
Then for each $g$ in the boundary support of $c$ we must have 
$$||\rho^t(g)-\rho_0(g)||<\frac{1-\varepsilon}{24}.$$
Applying Lemma~\ref{anal} part 2 this gives us
$$||\rho^t(g)^{-1}-\rho_0(g)^{-1}||<\frac{1-\varepsilon}{12}.$$
Then for $(a_i,b_i)$ in the support of $c$ we have 
\begin{align*}
||\rho^t(a_ib_i)\rho^t(b_i)^{-1}\rho^t(a_i)^{-1}-1||&<||\rho^t(a_ib_i)\rho^t(b_i)^{-1}\rho^t(a_i)^{-1}-\rho_0(a_ib_i)\rho_0(b_i)^{-1}\rho_0(a_i)^{-1}||+\varepsilon\\
&\le1.
\end{align*}
The last step is using Lemma~\ref{anal} part 1 with ``$N$''=3, ``$M$''=2, and ``$\varepsilon$''=$\frac{1-\varepsilon}{12}$. Applying this to $t=1$ we get that 

(2) The function
$$t\mapsto\openp\rho^t,c\closep$$
is a continuous function from $[0,1]$ to $\Z$ so it must be constant.
\end{proof}
\begin{thm}\label{obstruction}
If $\rho_0$ is a genuine representation, $\rho_1$ is a function from~$\Gamma$ to $U(n)$, $c$ is a 2-cocycle on $\Gamma$ and $$||\rho_1(g)-\rho_0(g)||<\frac{1}{24}$$
for all $g$ in the boundary support of $c$ then $\rho_1$ is 1-multiplicative on the support of $c$ and $\openp\rho_1,c\closep=0$.
\end{thm}
\begin{proof}
This follows from Proposition~\ref{homotop}, taking the limit as $\varepsilon\rightarrow0$.
\end{proof}

This is sufficient for our purposes, but for conceptual clarity it would be nice to show that this pairing depends only on homology class, not on the choice of cycle representative. We do not have a result that is quite this strong, but we can show that it is ``eventually true'' for asymptotic homomoprhisms.

\begin{thm}
Let $c=\partial d$ be a 2-boundary in $C^2(\Gamma)$ and let $\rho_n:\Gamma\rightarrow U(N_n)$ be an asymptotic homomorphism. Then for large enough $n$ we have $\openp\rho_n,c\closep=0$
\end{thm}
\begin{proof}
By linearity we may reduce the case that $c=\partial[g_1|g_2|g_3]$. In this case we have that
$$c=-[g_1|g_2]+[g_1|g_2g_3]-[g_1g_2|g_3]+[g_2|g_3].$$
For large enough $n$ we have that $\rho_n$ is multiplicative enough that we may apply Lemma~\ref{anal2}. Thus
\begin{align*}
2\pi i\openp\rho_n,c\closep=&\Tr(\log(\rho_n(g_1)\rho_n(g_2)\rho_n(g_1g_2)^{-1}\rho_n(g_1g_2)\rho_n(g_3)\rho_n(g_1g_2g_3)^{-1}\\
&\cdot\rho_n(g_1g_2g_3)\rho_n(g_2g_3)^{-1}\rho_n(g_1)^{-1}\rho_n(g_2g_3)\rho_n(g_3)^{-1}\rho_n(g_2)^{-1}))\\
=&\Tr(\log(\rho_n(g_1)\rho_n(g_2)\rho_n(g_3)\rho_n(g_2g_3)^{-1}\rho_n(g_1)^{-1}\rho_n(g_2g_3)\rho_n(g_3)^{-1}\rho_n(g_2)^{-1})).
\end{align*}
Now since the complex unitary group is path connected we can make a path $u_t:[0,1]\rightarrow U(n)$ so that $u_0=\rho(g_1)$ and $u_1=\Id_{\C^n}$. Then
\begin{align*}
||u_t\rho_n(g_2)\rho_n(g_3)\rho_n(g_2g_3)^{-1}u_t^{-1}-\Id_{\C^n}||&=||\rho_n(g_2)\rho_n(g_3)\rho_n(g_2g_3)^{-1}-u_t^{-1}u_t||\\
&=||\rho_n(g_2)\rho_n(g_3)\rho_n(g_2g_3)^{-1}-\Id_{\C^{N_n}}||.
\end{align*}
For large enough $n$ this will be less than small, so $$\log(u_t\rho_n(g_1)\rho_n(g_2)\rho_n(g_2g_3)^{-1}u_t^{-1}\rho_n(g_2g_3)\rho_n(g_3)^{-1}\rho_n(g_2)^{-1})$$
will be well-defined. Moreover since 
$$\det(u_t\rho_n(g_2)\rho_n(g_3)\rho_n(g_2g_3)^{-1}u_t^{-1}\rho_n(g_2g_3)\rho_n(g_3)^{-1}\rho_n(g_2)^{-1})=1$$
we must have that
$$\Tr(\log(u_t\rho_n(g_2)\rho_n(g_3)\rho_n(g_2g_3)^{-1}u_t^{-1}\rho_n(g_2g_3)\rho_n(g_3)^{-1}\rho_n(g_2)^{-1})\in2\pi i\Z.$$
Because this is a discrete space the values cannot depend on $t$. We conclude that
\begin{align*}
2\pi i\openp\rho_n,c\closep&=\Tr(\log(\rho_n(g_2)\rho_n(g_3)\rho_n(g_2g_3)^{-1}\rho_n(g_2g_3)\rho_n(g_3)^{-1}\rho_n(g_2)^{-1})\\
&=0.
\end{align*}
\end{proof}

\subsection{Finite Type Cohomology}

\begin{defn}\label{finiteT}
Let $\Gamma$ be a countable discrete group and $[\sigma]\in H^2(\Gamma;\Z)$ be given by the central extension
$$\begin{tikzcd}
e\arrow{r} & \Z
\arrow{r}{\iota} &
\Tilde{\Gamma} \arrow{r}{\varphi}
& \Gamma\arrow{r}{}& e.
\end{tikzcd}$$
We say that $[\sigma]$ is {\em finite type} if $\tilde{\Gamma}$ has a sequence of finite index subgroups $\{\Gamma_k\}_{k\in\N}$ so that
$$\iota(\Z)\cap\bigcap_{k}\Gamma_k=\{e\}.$$
\end{defn}

\begin{rmk}\label{rfinite}
Clearly if $\tilde{\Gamma}$ is residually finite $[\sigma]$ is of finite type.
\end{rmk}

\begin{rmk}\label{normalD}
We may assume that the subgroups in Definition~\ref{finiteT} are normal and decreasing. To achieve normality replace $\Gamma_k$ with the kernel of the action of $\tilde{\Gamma}$ on $\tilde{\Gamma}/\Gamma_k$. To achieve a decreasing sequence replace $\Gamma_k$ with the cumulative intersection of $\Gamma_k$.
\end{rmk}

To develop our formula we will develop an alternate characterization of finite type cohomology classes that can be expressed in terms of the cohomology cochain complex. 
\medskip

Let $\Gamma$ be a discrete group and let $Q$ be a finite quotient of $\Gamma$. Call~$q$ the quotient map from $\Gamma$ to $Q$ and $f^n$ the canonical map from $\Z$ to $\Z/n\Z$. Then $q$ induces a cochain map $q^\#$ from $C^k(Q;\Z/n\Z)$ to $C^k(\Gamma;\Z/n\Z)$ and $f$ induces a map $f_\#^n$ from $C^k(\Gamma;\Z)$ to $C^k(\Gamma;\Z/n\Z)$. These in turn induce maps $q^*:H^*(Q;\Z/n\Z)\rightarrow H^k(\Gamma;\Z/n\Z)$ and $f^n_*:H^*(\Gamma;\Z)\rightarrow H^k(\Gamma;\Z/n\Z)$.

\begin{defn}
If $\sigma$ is a $\Z$-valued $k$-cocycle on $\Gamma$, we say that $\sigma$ is of $n$-$Q$ {\em type} if $f^n_\#(\sigma)=q^\#(\sigma')$ for some $\Z/n\Z$-valued $k$-cocycle on $Q$, $\sigma'$. We say that a cohomology class~$[\sigma]\in H^k(\Gamma;\Z)$ is $n$-$Q$ {\em type} if there is a cohomology class $[\sigma']\in H^k(Q;\Z/n\Z)$ so that $f^n_*([\sigma])=q^*([\sigma'])$.
\end{defn}
$$\begin{tikzcd}
&
\left[\sigma'\right]\in H^k(Q;\Z/n\Z)\arrow{d}{q^*}
\\
\left[\sigma\right]\in H^k(\Gamma;\Z)\arrow{r}{f^n_*}
& H^k(\Gamma;\Z/n\Z).
\end{tikzcd}$$

\begin{ex}\label{z2cos}
Consider $\Gamma=\Z^2$ and the cocycle $\sigma((x_1,x_2),(y_1,y_2))=x_2y_2$. That this is a cocycle is easy to check:
\begin{align*}
\partial\sigma((x_1,x_2),(y_1,y_2),(z_1,z_2))&=x_2y_1-x_2(y_1+z_1)+(x_2+y_2)z_1-y_2z_1\\
&=0.
\end{align*}
This is not a coboundary because 
$$c=[(0,1)|(1,0)]-[(1,0)|(0,1)]$$
is a 2-chain such that
$$\langle\sigma,c\rangle=1.$$
Then if $Q_n=(\Z/n\Z)^2$, and let $q_n:\Z^2\rightarrow Q_n$ be the obvious quotient map. We have that $\sigma$, and hence $[\sigma]$, is of $n$-$Q_n$ type. To show this note that the same formula used for $\sigma$ defines a 2-cochain, $\sigma'\in C^2(Q_n;\Z/n\Z)$. The same computations that show that $\sigma$ is a cocycle also show that $\sigma'$ is a cocycle, and clearly $q_n^\#(\sigma)=f^n_\#(\sigma')$, where $f^n$ is the quotient map from $\Z$ to $\Z/n\Z$.
\end{ex}

\begin{prop}\label{finiteQ}
Suppose that $\Gamma$ is a discrete group and $[\sigma]\in H^2(\Gamma;\Z)$. Let the central extension corresponding to $[\sigma]$ be as follows:
\begin{equation*}
\begin{tikzcd}
e\arrow{r} & \Z
\arrow{r}{\iota} &
\Tilde{\Gamma} \arrow{r}{\varphi}
& \Gamma\arrow{r}{}& e.
\end{tikzcd}
\end{equation*}
The following are equivalent:
\begin{enumerate}
    \item $[\sigma]$ is of finite type;
    \item there are infinitely many $n\in\N$ so that there is a finite quotient of $\tilde{Q}_n$ of $\Tilde{\Gamma}$ so that $\iota(1)$ has order $n$ in the quotient;
    \item there are infinitely many $n\in\N$ so that $\Gamma$ has a finite quotient $Q_n$ so that $[\sigma]$ is of $n$-$Q_n$ type.
\end{enumerate}
\end{prop}
\begin{proof}
(1)$\Longrightarrow$(2) Let $\Gamma_k$ be a sequence of subgroups as in Definition~\ref{finiteT} and assume that they are normal as in Remark~\ref{normalD}. For $\ell\in\N$ pick~$\Gamma_k$ so that $\iota(1)^{\ell!}\not\in\Gamma_k$. Then let $n$ be the order of $\iota(1)$ in $\Tilde{\Gamma}/\Gamma_k$. Call $\tilde{Q}_n=\Tilde{\Gamma}/\Gamma_k$. Note that $n>\ell$ so letting $\ell\rightarrow\infty$ we get the desired family of subgroups for infinitely many distinct $n\in\N$.
\medskip

\noindent
(2)$\Longrightarrow$(1) Let $\Gamma_n$ be the kernel of the map from $\tilde{\Gamma}$ to $\tilde{Q}_n$. Then for any $\ell\in\Z\setminus\{0\}$ we can pick $n>|\ell|$ so that $\iota(1)^\ell\not\in\Gamma_n$.
\medskip

\noindent
(2)$\Longrightarrow$(3) By assumption, for infinitely many $n$ we have a diagram:

\begin{equation*}\begin{tikzcd}
e\arrow{r} & \Z
\arrow{r}{\iota}\arrow{d}{f^n} &
\Tilde{\Gamma}\arrow{d}{\Tilde{q}_n} \arrow{r}{\varphi}
& \Gamma\arrow{r}{}\arrow{d}{q_n}& e
\\
e\arrow{r} & \Z/n\Z
\arrow{r}{} &
\Tilde{Q}_n \arrow{r}{\varphi_n'}
& Q_n\arrow{r}{}& e
\end{tikzcd}\tag{3.1}\end{equation*}
Then we may factor the vertical maps as follows:
\begin{equation*}\begin{tikzcd}
e\arrow{r} & \Z
\arrow{r}{\iota}\arrow{d}{f^n} &
\Tilde{\Gamma}\arrow{d}{} \arrow{r}{\varphi}
& \Gamma\arrow{r}{}\arrow{d}{\Id}& e
\\
e\arrow{r} & \Z/n\Z\arrow{d}{\Id}
\arrow{r}{} &
\Tilde{\Gamma}/{\langle\iota(1)^n\rangle}\arrow{d} \arrow{r}{}
& \Gamma\arrow{r}{}\arrow{d}{q_n}& e\\
e\arrow{r} & \Z/n\Z
\arrow{r}{} &
\Tilde{Q}_n \arrow{r}{\varphi_n'}
& Q_n\arrow{r}{}& e
\end{tikzcd}\tag{3.2}\end{equation*}
By \cite[Chapter IV.3 Exercise 1]{coho}, $f^n_*([\sigma])$ corresponds to the middle row of diagram (3.2). Let $[\sigma']$ be the element of $H^2(Q_n;\Z/n\Z)$ corresponding to the bottom row of diagram (3.2). Again using \cite[Chapter IV.3 Exercise 1]{coho}  we get that $q_n^*([\sigma'])$ corresponds to the middle row as well. It follows that $f^n_*([\sigma])=q_n^*([\sigma'])$. 
\medskip

\noindent
(3)$\Longrightarrow$(2) Using \cite[Chapter IV.3 Exercise 1]{coho} as we did for the other direction we get that diagram (3.2) must exist for infinitely many $n$. Composing the vertical maps we get that quotients such as diagram (3.1) must exist for infinitely many $n$ as well.
\end{proof}

This theorem motivates a definition that extends the finite type concept to cocycle representatives of the cohomology class.

\begin{defn}
If $\sigma$ is a $\Z$-valued $2$-cocycle on $\Gamma$, we say that $\sigma$ is of {\em finite type} if for infinitely many $n\in\N$ there is a finite quotient $Q_n$ of $\Gamma$ so that that $\sigma$ is of $n$-$Q_n$ type. 
\end{defn}

The cocycle defined in Example~\ref{z2cos} is finite type. From Proposition~\ref{finiteQ} it follows that its cohomology class is as well.

\begin{rmk}\label{representative}
Obviously if $\sigma$ is of $n$-$Q$ type then so is $[\sigma]$. Conversely, if $[\sigma]\in H^k(\Gamma;\Z)$ is of $n$-$Q$ type then there exists some $\omega\in[\sigma]$ that is of $n$-$Q$ type. To show this let $q:\Gamma\rightarrow Q$ be the quotient map and let $f:\Z\rightarrow\Z/n\Z$ be the usual map. We can take $\dot\omega\in f_*([\sigma])$ so that $\dot\omega\in\im(q^\#)$. Let $\dot\omega=\partial\dot\alpha+f_\#(\sigma)$ where $\dot\alpha\in C^{1}(\Gamma;\Z/n\Z)$. Then find $\alpha\in C^{1}(\Gamma;\Z)$ so that $f_\#(\alpha)=\dot\alpha$. Then letting $\omega=\partial\alpha+\sigma$ we see that $f_\#(\omega)=\dot\omega\in\im(q^\#)$. Thus $\omega$ is of $n$-$Q$ type. However, a finite type cohomology class does {\em not} obviously have a finite type representative, because the choice of representative of~$[\sigma]$ might depend on~$n$.
\end{rmk}

\subsection{Constructing Asymptotic Representations From Cocycles}

We start by defining our formula for an asymptotic representation.

\begin{prop}\label{rhon}
Suppose that $\Gamma$ is a discrete group and $[\sigma]\in H^2(\Gamma;\Z)$. Let $Q_n$ be a finite quotient of $\Gamma$ so that $[\sigma]$ is of $n$-$Q_n$ type; this exists by Remark~\ref{representative}. Let $\sigma_n$ be a representative of $[\sigma]$ so that $\sigma_n$ is of $n$-$Q_n$ type. Let $\alpha_n$ be a 1-cochain so that $\sigma_n+\partial\alpha_n=\sigma$. Let $\hat{\chi}_n(g_1,g_2)=\exp\left(\frac{2\pi i}{n}(\alpha_n(g_1)+\sigma_n(g_1,g_2))\right)$. Then define $V_n=\ell^2(Q_n)$. Treat $\bar{g}$ as a basis element for $V_n$ where $g\in\Gamma$ and $\bar{g}$ is its image in $Q_n$. Then there is a well-defined function $\rho_n:\Gamma\rightarrow U(V_n)$ that obeys the formula
$$\rho_n(g_1)\Bar{g}_2=\hat{\chi}_n(g_1,g_2)\Bar{g}_1\Bar{g}_2.$$
\end{prop}
\begin{proof}
We will show that $\rho_n$ is well-defined. Suppose that $\bar{g}_2=\bar{g_2'}$. If we show that
$$\sigma_n(g_1,g_2)\equiv\sigma_n(g_1,g_2')\mod n$$
we will have shown that $\rho_n$ is well-defined because $\hat{\chi}_n$ only depends on $\sigma_n$ up to equivalence mod $n$. If $\dot\sigma_n$ is $\sigma_n$ reduced mod $n$ then we have by assumption that $\dot\sigma_n=q^\#(\sigma_n')$ where $q$ is the quotient map from $\Gamma$ to $Q$ and $\sigma_n'$ is a 2-cocycle in $C^2(Q;\Z/n\Z)$. Thus
$$\dot\sigma_n(g_1,g_2)=\sigma_n'(\bar{g}_1,\bar g_2)=\sigma_n'(\bar g_1,\bar{g_2'})=\dot{\sigma}_n(g_1,g_2').$$

Note that $\rho_n(g)$ maps the orthonormal basis $\{h:h\in Q_n\}$ to another orthonormal basis, so it is a unitary.
\end{proof}

\begin{defn}
Suppose $\rho:\Gamma\rightarrow U(k)$ and $\chi$ is an $S^1$ valued 2-cocycle on $\Gamma$. If $\rho$ obeys the formula in $\rho(gh)\rho(h)^{-1}\rho(g)^{-1}=\chi(g,h)\Id_{\C^n}$  is called a {\em projective representation} with cocycle $\chi$.
\end{defn}

We will show that $\rho_n$ is a projective representation in Lemma~\ref{chiform}.

In the case that $\alpha_n=0$ our formula reduces to what is known as the {\em projective left regular representation} for $Q_n$ and $\chi_n$, for example see~\cite{factorsfrom} page 2. The proof of Corollary~B in~\cite{llp} also uses a cohomology class that ``behaves well'' with respect to finite quotients, to make projective representations.

An alternate justification for the formula is as follows. Suppose that there is a finite quotient of a central extension of $\Gamma$ as follows
\begin{equation*}\begin{tikzcd}
e\arrow{r} & \Z
\arrow{r}{\iota}\arrow{d}{f^n} &
\Tilde{\Gamma}\arrow{d}{\Tilde{q}_n} \arrow[shift left]{r}{\varphi}
& \Gamma\arrow{r}{}\arrow{d}{q_n}\arrow[shift left]{l}{\theta}& e
\\
e\arrow{r} & \Z/n\Z
\arrow{r}{\iota'} &
\Tilde{Q}_n \arrow{r}{\varphi_n'}
& Q_n\arrow{r}{}& e
\end{tikzcd}\end{equation*}
and a set theoretic section $\theta$ of the extension. Let $\pi_n$ be the induced representation of $\tilde{Q}_n$ from the character on $\iota'(\Z/n\Z)$ that takes $\iota'(1)\mapsto\exp(2\pi i/n)$. Then $\rho_n=\pi_n\circ\tilde q_n\circ\theta$. Deriving the formula from here is technical.

The discussion at the start of Chapter~3.3 in~\cite{projectiveRep} explains how one should expect a projective representation to come from a splitting and representation of $\tilde\Gamma$ as described above.

\begin{rmk}\label{2ndEntry}
The existence technically only uses the fact that $\dot\sigma(g_1,g_2)$ depends only on $g_1$ and the reduction of $g_2$ in $Q_n$, rather than the image of both $g_1$ and $g_2$ in $Q_n$. We will use this fact in examples to reduce the asymptotics of the dimension of the asymptotic representation.
\end{rmk}

\begin{lemma}\label{chiform}
Let $\Gamma$, $n$, $\sigma$, $Q_n$, and $\rho_n$ be as above. Define $\chi_n\in C^2(\Gamma;S^1)$ by
$$\chi_n(g_1,g_2)=\exp\left(\frac{2\pi i }{n}\sigma(g_1,g_2)\right).$$
Then $\rho_n$ obeys the formula
$$\rho_n(g_1g_2)\rho_n(g_2)^{-1}\rho_n(g_1)^{-1}=\chi_n(g_1,g_2)^{-1}\Id_{V_n}.$$
\end{lemma}
\begin{proof}
Define $\hat{\sigma}_n(g_1,g_2)=\alpha_n(g_1)+\sigma_n(g_1,g_2)$. We compute
\begin{align*}
-\hat{\sigma}_n(g_1,&g_2g_3)+\hat{\sigma}_n(g_1g_2,g_3)-\hat{\sigma}_n(g_2,g_3)\\
&=\alpha_n(g_1g_2)-\alpha_n(g_1)-\alpha_n(g_2)-\sigma_n(g_1,g_2g_3)+\sigma_n(g_1g_2,g_3)-\sigma_n(g_2,g_3)\\
&=-\partial\alpha_n(g_1,g_2)-\sigma_n(g_1,g_2)\\
&=-\sigma(g_1,g_2).&({3.3})
\end{align*}
The second equality follows from the fact that $\sigma_n$ is a cocycle. Exponentiating both sides of (3.3) we get
\begin{align*}
\hat{\chi}_n(g_1,g_2g_3)^{-1}\hat{\chi}_n(g_1g_2,g_3)\hat{\chi}_n(g_2,g_3)^{-1}=\chi_n(g_1,g_2)^{-1}&&(3.4)
\end{align*}
Next we claim that
$$\rho_n(g_1)^{-1}\bar{g}_2=\hat{\chi}_n(g_1,g_1^{-1}g_2)^{-1}\bar{g}_1^{-1}\bar{g}_2.$$
To check this it suffices to compute that
\begin{align*}
\rho_n(g_1)\hat{\chi}_n(g_1,g_1^{-1}g_2)\bar{g}_1^{-1}\bar{g}_2&=\hat{\chi}_n(g_1,g_1^{-1}g_2)^{-1}\rho_n(g_1)\bar{g}_1^{-1}\bar{g}_2\\
&=\hat{\chi}_n(g_1,g_1^{-1}g_2)^{-1}\hat{\chi}_n(g_1,g_1^{-1}g_2)\bar{g}_2\\
&=\bar{g}_2.
\end{align*}
Using this we can compute
\begin{align*}
\rho_n(g_1g_2)\rho_n(g_2)^{-1}\rho_n(g_1)^{-1}\Bar{g}_3&=\rho_n(g_1g_2)\rho_n(g_2)^{-1}\hat\chi_n(g_1,g_1^{-1}g_3)^{-1}\Bar{g}_1^{-1}\Bar{g}_3\\
&=\rho_n(g_2g_2)\hat\chi_n(g_2,g_2^{-1}g_1^{-1}g_3)^{-1}\hat\chi_n(g_1,g_1^{-1}g_3)^{-1}\Bar{g}_2^{-1}\Bar{g}_1^{-1}\Bar{g}_3\\
&=\hat\chi_n(g_1g_2,g_2^{-1}g_1^{-1}g_3)\hat\chi_n(g_2,g_2^{-1}g_1^{-1}g_3)^{-1}\hat\chi_n(g_1,g_1^{-1}g_3)^{-1}\Bar{g}_3\\
&=\chi_n(g_1,g_2)^{-1}\Bar{g}_3.
\end{align*}
Here the last step follows from (3.4) applied to $g_1$, $g_2$ and $g_2^{-1}g_2^{-1}g_3$. 
\end{proof}

Now we are ready to prove Theorem~\ref{coc2}

\begin{proof}
First note that by Proposition~\ref{finiteQ} and Remark~\ref{representative} there are infinitely many $n$ so the formula given in Proposition~\ref{rhon} is well-defined.

Now we will show that $\rho_n$ is asymptotically multiplicative. Noting that since $\sigma$ does not depend on $n$ we have that $\chi_n(g_1,g_2)$, defined as in Lemma~\ref{chiform}, goes to 1 as~$n$ goes to infinity. Thus Lemma~\ref{chiform} implies asymptotic multiplicativity.

Now we will show that for large enough $n$, $\rho_n$ is not close to any genuine representation of $\Gamma$ on a particular finite subset of $\Gamma$. From the fact that $[\sigma]\not\in\ker(h)$ there is some 2-cycle $c\in C_2(\Gamma)$ written
$$c=\sum_{i=1}^Nx_i[a_i|b_i].$$
so that
$$\langle \sigma,c\rangle\ne0.$$
Then we compute that
\begin{align*}
\openp\rho_n,c\closep&=\frac{1}{2\pi i}\sum_{j=1}^N x_j\Tr(\log(\rho_n(a_jb_j)\rho_n(b_j)^{-1}\rho_n(a_j)^{-1}))\\
&=\frac{1}{2\pi i}\sum_{j=1}^N x_j\Tr(\log(\chi_j(a_jb_j)^{-1})\Id_{V_n})&\mbox{by Lemma~\ref{chiform}}\\
&=-\frac{1}{2\pi i}\sum_{j=1}^N x_j\frac{2\pi i}{n}\sigma(a_j,b_j)\Tr(\Id_{V_n})\\
&=-\langle\sigma,c\rangle\frac{\dim(V_n)}{n}\\
&\ne0.
\end{align*}
By Theorem~\ref{obstruction} it follows that $\rho_n$ cannot be within $\frac{1}{24}$ of a genuine representation on the boundary support of $c$, and thus cannot be perturbed to a genuine representation.
\end{proof}
\begin{rmk}\label{finitelyGen}
If $\Gamma$ is finitely generated then pairing nontrivially with a cohomology class is equivalent to being non-torsion. To see this note that from the universal coefficient theorem \cite[Theorem~53.1]{eat}  we have a short exact sequence
$$\begin{tikzcd}
0\arrow{r}{}&\Ext(H_1(\Gamma),\Z)\arrow{r}{}& H^2(\Gamma;\Z)\arrow{r}{}&\Hom(H_2(\Gamma),\Z)\arrow{r}{}&0
\end{tikzcd}.$$
Clearly any torsion element cannot fit the condition since $\Hom(H_2(\Gamma),\Z)$ is a torsion-free group. To see the converse note that since $\Gamma$ is finitely generated $H_1(\Gamma)\cong\Gamma/[\Gamma,\Gamma]$ \cite[page 36]{coho} is finitely generated as well. Thus $\Ext(H_1(\Gamma),\Z))$ can be show to be torsion from \cite[Theorem 52.3]{eat} and the table on \cite{eat} page~331.
\end{rmk}

\begin{rmk}\label{simplify}
In many examples we will have the stronger condition that there is a particular cocycle representative $\sigma$ of $[\sigma]$ that is of finite type. In this case the formula simplifies to
$$\rho_n(g_1)\bar{g}_2=\chi_n(g_1,g_2)\bar{g}_1\bar{g}_2$$
where $\chi_n(g_1,g_2)=\exp\left(\frac{2\pi i}{n}\sigma(g_1,g_2)\right)$.
\end{rmk}

\begin{cor}\label{polycresult}
Suppose that $\Gamma$ is a virtually polycyclic group with non-torsion 2-cohomology. Then $\Gamma$ meets the conditions of Theorem~\ref{coc2}, and is thus not matricially stable.
\end{cor}
\begin{proof}
First we note that if $[\sigma]$ is a non-torsion cohomology class it is not in the kernel of the map from $H^2(\Gamma;\Z)$ to $\Hom(H_2(\Gamma),\Z)$ by Remark~\ref{finitelyGen}. Let
\begin{equation*}
\begin{tikzcd}
e\arrow{r} & \Z
\arrow{r}{\iota} &
\Tilde{\Gamma} \arrow{r}{\varphi}
& \Gamma\arrow{r}{}& e
\end{tikzcd}
\end{equation*}
be the central extension corresponding to $[\sigma]$. We have that $\Tilde{\Gamma}$ is also virtually polycyclic by Proposition~\ref{virtualPext}. Now because $\tilde{\Gamma}$ is virtually polycyclic it is residually finite \cite[Theorem 3]{Hirsch}. Thus $[\sigma]$ is of finite type. 
\end{proof}

\section{Torsion-Free Finitely Generated Nilpotent Groups}

The purpose of this section is to provide an alternate proof that torsion-free finitely generated nilpotent groups fit the conditions of Theorem~\ref{coc2} (Theorem~\ref{nil}). While this follows from Corollary~\ref{polycresult} the alternate proof gives rise to a simple formula for the asymptotic representation.

\begin{prop}\label{mal}
Suppose that $\Gamma$ is a torsion-free finitely generated nilpotent group with a Mal'cev basis $(a_1,\ldots, a_m)$ and a central extension as follows:
$$e\begin{tikzcd}\arrow{r} & \Z
\arrow{r}{\iota} &
\Tilde{\Gamma} \arrow{r}{\varphi}
& \Gamma\arrow{r}{}& e.
\end{tikzcd}$$
Then if $\Tilde{a}_i$ is a lift of $a_i$ for $i\in\{1,\ldots,m\}$ and $\Tilde{a}_{m+1}=\iota(1)$, $(\Tilde{a}_1,\ldots,\Tilde{a}_{m+1})$ is a Mal'cev basis for $\Tilde{\Gamma}$. 
\end{prop}
\begin{proof}
Let $\theta$ be the set theoretic section for $\varphi$ defined by $\theta(a_1^{x_1}\cdots a_m^{x_m})=\Tilde{a}_1^{x_1}\cdots\Tilde{a}_m^{x_m}$. First we claim that any element $g\in\Tilde{\Gamma}$ can be written in the form $\Tilde{a}_1^{x_1}\cdots\Tilde{a}_m^{x_m}\cdot\Tilde{a}_{m+1}^{x_{m+1}}$. To see this note that $g=(\theta\circ\varphi(g))\Tilde{a}_{m+1}^{x_{m+1}}$. Then the claim follows from the definition of~$\theta$. Next to show that this is unique we suppose that $\Tilde{a}_1^{x_1}\cdots \Tilde{a}_{m+1}^{x_{m+1}}=\Tilde{a}_1^{y_1}\cdots\Tilde{a}_{m+1}^{y_{m+1}}$. Noting that
$$a_1^{x_1}\cdots a_m^{x_m}=\varphi(\Tilde{a}_1^{x_1}\cdots \Tilde{a}_{m+1}^{x_{m+1}})=\varphi(\Tilde{a}_1^{y_1}\cdots \Tilde{a}_{m+1}^{y_{m+1}})=a_1^{y_1}\cdots a_m^{y_m}$$
we get that $x_i=y_i$ for $i\ne m+1$. Then equality for $i=m+1$ follows canceling the other terms and noting that $\Tilde{a}_{m+1}$ is non torsion. Next we define $\Tilde{\Gamma}_i=\langle \Tilde a_i,\ldots, \Tilde a_{m+1}\rangle$ and similarly $\Gamma_i=\langle a_i,\ldots ,a_m\rangle$. By our assumptions $\Gamma_i$ form a central series for $\Gamma$. Note that for $i\le m$
\begin{align*}
[\Tilde{\Gamma},\Tilde{\Gamma}_i]&\subseteq\varphi^{-1}(\varphi([\Tilde{\Gamma},\Tilde{\Gamma}_i]))\\
&\subseteq\varphi^{-1}([\Gamma,\Gamma_i])\\
&\subseteq\varphi^{-1}(\Gamma_{i+1})\\
&\subseteq\langle\Tilde{a}_{m+1},\Tilde{\Gamma}_{i+1}\rangle\\
&=\Tilde{\Gamma}_{i+1}.
\end{align*}
For $i=m+1$ we have $[\Tilde{\Gamma},\Tilde{\Gamma}_{m+1}]=\{e\}=\Tilde{\Gamma}_{m+2}$. This shows that $\Tilde{\Gamma}_i$ is a central series which completes our proof.
\end{proof}

\begin{thm}\label{nil}
Suppose that $\Gamma$ is a torsion-free finitely generated nilpotent group that is not $\Z$ or trivial. Then $\Gamma$ has a cohomology class that meets the conditions of Theorem~\ref{coc2}, and the asympotic representation can be expressed as follows: $\Gamma$ can be viewed as $\Z^m$ with the multiplication as follows
$$\mathbf{x}*\mathbf{y}=(\eta_1(\mathbf{x,y}),\ldots,\eta_m(\mathbf{x,y}))$$
where $\eta_1,\ldots,\eta_n$ are rational\footnote{The polynomials will always take integer values if given integer inputs, but in general the coefficients may not be integers.}  polynomials in $\mathbf{x}=(x_1,\ldots, x_m)$ and $\mathbf{y}=(y_1,\ldots,y_m)$. In addition we have a non-torsion cocycle $\sigma(\mathbf{x,y})$ that is also a rational polynomial in the entries. Then the underlying vector space is $(\C^n)^{\tensor m}$. Then for $n$ co-prime to the denominators of coefficient in the $\eta_i$ and $\sigma$ we have
$$\rho_n(\mathbf{x})e_{y_1}\tensor e_{y_2}\tensor\cdots\tensor e_{y_m}=\exp\left(\frac{2\pi i}{n}\sigma(\mathbf{x,y})\right)e_{\eta_1(\mathbf{x,y})}\tensor\cdots\tensor e_{\eta_m(\mathbf{x,y})}$$
where we have the convention that $e_{j+n}=e_j$.
\end{thm}
\begin{proof}
By Corollary~\ref{nilpairing} there is a cohomology class $[\sigma]\in H^2(\Gamma;\Z)$ which pairs nontrivially with a homology class. By \cite[Theorem IV.3.12]{coho} we have a central extension corresponding to $[\sigma]$:
$$\begin{tikzcd}
e\arrow{r} & \Z
\arrow{r}{} &
\Tilde{\Gamma} \arrow{r}{\varphi}
& \Gamma\arrow{r}{}& e.
\end{tikzcd}$$
By Proposition~\ref{mal} we have a Mal'cev basis $(\Tilde{a}_1,\ldots,\Tilde{a}_{m+1})$ for $\tilde\Gamma$ and a Mal'cev basis $(a_i,\ldots, a_m)$ for $\Gamma$ where for $i\ne m+1$ we have $\varphi(\tilde{a}_i)=a_i$. Any element in~$\Tilde{\Gamma}$ can be written uniquely as $\Tilde a_1^{x_1}\cdots \tilde a_{m+1}^{x_{m+1}}$. It follows that
$$\Tilde a_1^{x_1}\cdots \Tilde a_{m+1}^{x_{m+1}}\cdot \tilde a_1^{y_1}\cdots  \tilde a_{m+1}^{y_{m+1}}=\tilde a_1^{z_1}\cdots \tilde{a}_{m+1}^{z_{m+1}}$$
where $z_i$ is a function of $x_1,\ldots, x_{m+1},y_1,\ldots y_{m+1}$. Hall has shown that these functions are rational polynomials \cite[Theorem 6.5]{hall}. They can be computed by methods here \cite{deep}. Then if we pick a section $\theta:a_1^{x_1}\cdots a_m^{x_m}\mapsto\tilde{a}_1^{x_1}\cdots\tilde{a}_m^{x_m}$ by \cite[Chapter IV.3 equation 3.3.]{coho} there is a cocycle $\sigma'$ with $[\sigma]=[\sigma']$ so that
$$\theta(g)\theta(h)=\theta(gh)\tilde{a}_{m+1}^{\sigma'(g,h)}.$$
Since $\theta(gh)$ is in canonical form and has no power of $\tilde a_{m+1}$ it follows that $\sigma'(g,h)$ must be a rational polynomial of the powers in the canonical forms of $g$ and~$h$. Note that by writing elements of $\Gamma$ in canonical form we get that $\Gamma$ can be viewed as $\Z^m$ with a multiplication given by polynomial formulas of the entries. Thus if we take $n$ to be co-prime to the denominator of each polynomial in the multiplication for $\Gamma$ and the denominator of $\sigma'$ we may define a quotient $Q_n$ of~$\Gamma$ by reducing each entry of $\Gamma$ mod $n$. Then we may also reduce the formula for $\sigma'$ mod $n$ showing that $\sigma'$ is of $n$-$Q_n$ type. To justify the formula note that $\ell^2(Q_n)\cong(\C^n)^{\otimes n}$ and the isomorphism sends $(y_1,\ldots, y_m)\mapsto e_{y_1}\tensor\cdots \tensor e_{y_m}$. The applying Proposition~\ref{rhon} we get the formula mentioned here, with $\sigma_n=\sigma$, and so $\alpha_n=0$.
\end{proof}

\begin{rmk}\label{dimReduce}
Additionally we may modify this construction to get an asymptotic representation of dimension $n^{m-1}$ instead of dimension $n^m$. Note that the power of $\tilde a_{m+1}$ in the product $(\tilde a_1^{x_1}\cdots \tilde a_m^{x_m})(\tilde a_1^{y_1}\cdots\tilde a_m^{y_m})$ can be computed by using the relations to put all terms in order. If we leave the~$\tilde a_m^{y_m}$ term at the end until the last step, we notice that we may have to switch the positions of the $\tilde a_m$ and $\tilde a_{m+1}$ but these commute by because the extension is central. From this it follows that $\sigma$ does not depend on $y_m$. Thus we may replace our quotient $Q_n$ with $Q_n'=Q_n/\langle q_n(a_m)\rangle$ and by Remark~\ref{2ndEntry} we may use the formula for $\rho_n$ except with $Q_n'$ instead of $Q_n$. The rest of the proof for asymptotic multiplicativity and non-perturbability flows the same way. An alternative explanation is that because $\sigma(\mathbf{x,y})$ does not depend on $y_m$ the formula for $\rho_n$ commutes with the projection $r_n=\Id_{\C^n}\tensor\cdots\tensor\Id_{\C^n}\tensor p_n$ where $p_n$ is defined by the formula $p_ne_i=\sum_{j}\frac{1}{\sqrt{n}}e_j$. Thus $r_n\rho_nr_n$ defines an $n^{m-1}$-dimensional asymptotic representation.
\end{rmk}

\section{Generators and Relations}

In this section we will show how to recognize conditions of Theorem~\ref{coc2} by looking at a presentation of a group. This can be helpful for computing the cocycle. It also relates our method to the method of other authors such as~\cite{voiculescuM}, \cite{KvoiculescuM}, and~\cite{stab}. 

Note that there must be some cocycle that satisfies the first condition in Theorem~\ref{coc2} if the 2-homology group is non-torsion. Note that because of Theorem~\ref{hopf} it is possible to view the 2-homology as a group of relations, so it should be possible to test this condition by looking at relations of the group. 

Let $\Gamma=\langle g_1,\ldots, g_m|r_1,\ldots\rangle$. By convention we will let $\mathbb{F}_m=\langle\gamma_1,\ldots,\gamma_m\rangle$. 

\begin{defn}
We say the relation $r_1$ is {\em homogeneous} if $r_1\in[\F_m,\F_m]$. We say that {\em powers of} $r_1$ {\em are centrally irredundant relative to $r_2,\ldots$} if $r_1^n$ is not in the normal subgroup of $\mathbb{F}_m$ generated by $r_2,\ldots$ and $ [\gamma_1,r_1],\ldots,[\gamma_m,r_1]$ for any $n\in\Z^+$.
\end{defn}

\begin{ex}\label{z2ext}
If $\Z^2$ is written as $\langle a,b|aba^{-1}b^{-1}=1\rangle$ then the relation $aba^{-1}b^{-1}$ is both homogeneous and powers of it are centrally irredundant. 
\end{ex}

Note that if $r_1$ meets both these conditions then it corresponds to a non torsion element of $H_2(\Gamma,\Z)$ by Theorem~\ref{hopf}. 

\begin{prop}\label{extension}
If $\Gamma=\langle g_1,\ldots, g_m|r_1,\ldots\rangle$ where $r_1$ is homogeneous and powers of~$r_1$ are centrally irredundant relative to the other relations then there is a central extension
\begin{equation*}
\begin{tikzcd}
1\arrow{r} & \Z
\arrow{r}{\iota} &
\Tilde{\Gamma} \arrow{r}{\varphi}
& \Gamma\arrow{r}{}& 1
\end{tikzcd}
\end{equation*}
by taking 
$$\Tilde{\Gamma}=\langle \Tilde{g}_1,\ldots, \Tilde{g}_m|[\Tilde{g}_1,r_1],\ldots,[\Tilde{g}_m,r_1], r_2,r_3,\ldots\rangle$$
and mapping $\iota:1\mapsto \Tilde{r}_1\in\Tilde{\Gamma}$ and $\varphi:\Tilde{g}_i\mapsto g_i$ where $\Tilde{r}_1$ is the image of $r_1$ in the map from $\F_m$ to $\Tilde{\Gamma}$ defined by mapping $\gamma_i\mapsto\Tilde{g}_i$.
\end{prop}
\begin{proof}
The map $\iota$ is injective because powers of $r_1$ are centrally irredundant. 

Because $r_1$ is one of the defining relations of $\Gamma$, we have $\im(\iota)\subseteq\ker(\varphi)$.

Note that $\im\iota$ is central in $\Tilde{\Gamma}$ by the construction of $\Tilde{\Gamma}$.

The map $\varphi$ is well defined because every relationship defining $\Tilde{\Gamma}$ is true of the generators of $\Gamma$ as well.

If $w$ is some word in the generators that is in $\ker(\varphi)$ then $w$ pulls back to some word in $\F_m$ that is in the smallest normal subgroup containing $r_1,\ldots$. Since every other relation is trivial in $\Tilde{\Gamma}$, this means $w$ is in the smallest normal subgroup containing $\tilde{r}_1$. However since $\tilde{r}_1$ is central, this group is precisely $\im(\iota)$.

The map $\varphi$ is surjective because it hits every generator of $\Gamma$.
\end{proof}

\begin{defn}
We call the central extension in Proposition~\ref{extension} the {\em dual extension of $r_1$ relative to the other relations}.
\end{defn}

\begin{ex}
If $\Z^2\cong\langle a,b|[a,b]=1\rangle$ is given the same presentation as in Example~\ref{z2ext} then the dual extension to the relation $aba^{-1}b^{-1}$ can be expressed as the group
$$H=\langle a,b,c|c=[a,b],[a,c]=[b,c]=1\rangle$$
commonly known as the discrete Heisenberg group. The map $\iota$ is given by $\iota(1)=c$ and $\varphi$ is determined by $\varphi(a)=a$ and $\varphi(b)=b$
\end{ex}
 
\begin{prop}\label{same}
Suppose that $r_1$ is a homogeneous relation on a group $\Gamma=\langle g_1,\ldots, g_m|r_1,\ldots\rangle$ so that powers of $r_1$ are centrally irredundant relative to the other relations. If~$[\sigma]$ is the 2-cohomology class corresponding to the dual extension of $r_1$ and $[c]$ is the 2-homology class corresponding to $r_1$ then $\langle\sigma,c\rangle=1$.
\end{prop}
\begin{proof}
Pick a section $\theta$ of the extension so that $\theta(1)=1$. Then $\sigma$ is defined by the formula
\[\Tilde{r}_1^{-\sigma(a,b)}=\theta(ab)\theta(b)^{-1}\theta(a)^{-1}.\tag{see \cite[Chapter IV.3 equation 3.3.]{coho}}\]
To compute the 2-homology class corresponding to $r_1$ we need write $r_1$ as
$$r_1=\prod_{i=1}^N[{a}_i,{b}_i]$$
then by \cite[chapter II.5 Exercise 4]{coho} the class corresponding to $r_1$ can be written as
\[c=\sum_{i=1}^N([I_{i-1}|\dot a_i]+[I_{i-1}\dot a_i|\dot b_i]-[I_{i-1}\dot a_i\dot b_i\dot a_i^{-1}|\dot a_i]-[I_i|\dot b_i])\tag{5.1}\]
where $\dot a_i$ and $\dot b_i$ are the images in $\Gamma$ of ${a}_i$ and ${b}_i$ and $I_i=[\dot a_1,\dot b_1]\cdots[\dot a_i,\dot b_i]$. Call $\Tilde{a}_i=\theta(\dot a_i)$ and $\Tilde{b}_i=\theta(\dot b_i)$. Here $\Tilde{a}_i$ and $\Tilde{b}_i$ must be the image of ${a}_i$ and ${b}_i$ in the map from $\F_m$ to $\Tilde{\Gamma}$ times some power of $\Tilde{r}_1$. Since $\Tilde{r}_1$ is central in $\Tilde{\Gamma}$ this power can be ignored while computing $[\Tilde{a}_i,\Tilde{b}_i]$. Thus
$$\prod_{i=1}^N[\Tilde{a}_i,\Tilde{b}_i]=\Tilde{r}_1.$$
Note that
$$I_{i-1}\dot{a}_i\dot b_i\dot a_i^{-1}=I_{i-1}\dot a_i\dot b_i\dot a_i^{-1}\dot b_i^{-1}b_i=I_i\dot b_i.$$
Let $c_i$ be the $i$th term in the sum of equation~(5.1). We get
\begin{align*}
\Tilde{r}_1^{-\langle\sigma,c_i\rangle}=&\theta(I_i)\theta(\dot b_i)\theta(I_i\dot b_i)^{-1}\\
&\cdot\theta(I_{i-1}\dot a_i\dot b_i\dot a_i^{-1})\theta(\dot a_i)\theta(I_{i-1}\dot a_i\dot b_i)^{-1}\\
&\cdot\theta(I_{i-1}\dot a_i\dot b_i)\theta(\dot b_i)^{-1}\theta(I_{i-1}\dot a_i)^{-1}\\
&\cdot\theta(I_{i-1}\dot a_i)\theta(\dot a_i)^{-1}\theta(I_{i-1})^{-1}\\
=&\theta(I_i)[\Tilde{b}_i,\Tilde{a}_i]\theta(I_{i-1})^{-1}.
\end{align*}
Then multiplying these in reverse order, starting from $i=N$ and ending at $i=1$ we get
\begin{align*}
\Tilde{r}_1^{-\langle\sigma,c\rangle}&=\theta(I_N)\left(\prod_{i=0}^{N-1}[\Tilde{b}_{N-i},\Tilde{a}_{N-i}]\right)\theta(I_0)^{-1}\\
&=\Tilde{r}_1^{-1}.
\end{align*}
Here the last equality is because $I_0=I_N=1$ since $I_N$ is the relation $r_1$.
\end{proof}

\begin{thm}\label{grmain}
Suppose that $r_1$ is a homogeneous relation on a group $\Gamma=\langle g_1,\ldots, g_m|r_1,\ldots\rangle$ so that powers of $r_1$ are centrally irredundant relative to the other relations. If 
\begin{equation*}
\begin{tikzcd}
e\arrow{r} & \Z
\arrow{r}{\iota} &
\Tilde{\Gamma} \arrow{r}{\varphi}
& \Gamma\arrow{r}{}& e
\end{tikzcd}
\end{equation*}
is the dual extension of $r_1$ relative to this presentation and there is a family of finite index subgroups $\Gamma_k\le\tilde\Gamma$ so that $\iota(\Z)\cap\bigcap_{k}\Gamma_k=\{1\}$. Then $\Gamma$ meets the conditions of Theorem~\ref{coc2}.
\end{thm}

\begin{proof}
This follows immediately from Proposition~\ref{same}.
\end{proof}

\section{Motivating Our Formula}

The purpose of this section is to explain the motivation behind the formula in Proposition~\ref{rhon}. For clarity recall the statement of the proposition:
\medskip

{\em Suppose that $\Gamma$ is a discrete group and $[\sigma]\in H^2(\Gamma;\Z)$. Let $Q_n$ be a finite quotient of $\Gamma$ so that $[\sigma]$ is of $n$-$Q_n$ type. Let $\sigma_n$ be a representative of $[\sigma]$ so that $\sigma_n$ is of $n$-$Q_n$ type. Let $\alpha_n$ be a 1-cochain so that $\sigma_n+\partial\alpha_n=\sigma$. Let $\hat{\chi}_n(g_1,g_2)=\exp\left(\frac{2\pi i}{n}(\alpha_n(g_1)+\sigma_n(g_1,g_2))\right)$. Then define $V_n=\ell^2(Q_n)$. Treat $\bar{g}$ as a basis element for $V_n$ where $g\in\Gamma$ and $\bar{g}$ is its image in $Q_n$. Then there is a well-defined function $\rho_n:\Gamma\rightarrow U(V_n)$ that obeys the formula}
$$\rho_n(g_1)\Bar{g}_2=\hat{\chi}_n(g_1,g_2)\Bar{g}_1\Bar{g}_2.$$

In terms of central extensions this means we have the following

\begin{equation*}\begin{tikzcd}
e\arrow{r} & \Z
\arrow{r}{\iota}\arrow{d}{f^n} &
\Tilde{\Gamma}\arrow{d}{\Tilde{q}_n} \arrow[shift left]{r}{\varphi}
& \Gamma\arrow{r}{}\arrow{d}{q_n}\arrow[shift left]{l}{\theta}& e
\\
e\arrow{r} & \Z/n\Z
\arrow{r}{\iota'} &
\Tilde{Q}_n \arrow[shift left]{r}{\varphi_n'}
& Q_n\arrow{r}{}\arrow[shift left]{l}{\theta_n'}& e
\end{tikzcd}\tag{6.1}\end{equation*}
Here we are taking $\theta$ to be the section from $\Gamma$ to $\tilde{\Gamma}$ corresponding to the cocycle~$\sigma$ and $\theta_n'$ to be a section from~$Q_n$ to $\tilde Q_n$. Neither is a group homomorphism. Consider the character on $\Z/n\Z$ defined by $1\mapsto\exp\left(\frac{2\pi i}{n}\right)$. Then if $\pi_n$ is the representation induced by this character and the inclusion $\iota'$ we will show that $\rho_n=\pi_n\circ\tilde q_n\circ \theta$. We may view this induced representation as the left regular representation of $Q_n$ ``twisted'' by the cocycle $\sigma_n$ which corresponds to a $\Z/n\Z$ cocycle on $Q_n$ because $\sigma_n$ is of $n$-$Q_n$ type. The $\alpha_n$ in the formula is there because the right to left maps in the diagram may not commute with the horizontal maps. We cannot always pick $\theta$ to commute with the horizontal maps in this way, because if we did~$\theta$ would need to depend on $n$ and our argument for asymptotic multiplicativity would not hold. The following theorem summarizes this result.

\begin{thm}\label{cext2}
Suppose that $\Gamma$ meets the conditions of Theorem~\ref{coc2}. Then our formula for an asymptotic representation that cannot be perturbed to a genuine representation can be obtained as follows. Consider the central extension corresponding to $[\sigma]$:
\begin{equation*}
\begin{tikzcd}
1\arrow{r} & \Z
\arrow{r}{\iota} &
\Tilde{\Gamma} \arrow{r}{\varphi}
& \Gamma\arrow{r}{}& 1.
\end{tikzcd}
\end{equation*}
Let $S$ be the set of $n\in\N$ so that $\tilde\Gamma$ has a finite quotient~$\tilde Q_n$ where $\iota(1)$ has order~$n$. Henceforth assume $n\in S$. Let~$\theta$ be a set theoretic section from $\Gamma$ to $\Tilde{\Gamma}$, and let $\tilde{q}_n$ be the quotient map from $\Tilde{\Gamma}$ to $\Tilde{Q}_n$. If~$\pi_n$ is the induced representation on $\Tilde{Q}_n$ by the character on $\langle \tilde{q}_n(\iota(1))\rangle$ defined by $\tilde{q}_n((\iota(1))\mapsto\exp\left(\frac{2\pi i}{n}\right)$. Then $\rho_n=\pi_n\circ \tilde q_n\circ\theta$ equivalent to the formula given in Proposition~\ref{rhon}. 
\end{thm}

\begin{proof}

Then let $\omega_n$ be the character on $\langle \tilde{q}_n(\iota(1))\rangle$ defined by $\tilde{q}_n(\iota(1))\mapsto \exp\left(\frac{2\pi i}{n}\right)$. By \cite[Exercise 2.3.16]{rept}, one construction for the induced representation comes from the space
$$\bigoplus_{i=1}^n g_i\C\cong V_n$$
where one element $g_i$ is a representative of a different one of the $n$ cosets in~$\tilde{Q}_n/\langle\iota(1)\rangle$ for each~$i$. Here $V_n$ is the space as described in Proposition~\ref{rhon}. Then the induced representation is defined by
$$\pi_n(g)g_i=\omega_n(\Tilde{q}_n(\iota(1))^k)g_j$$
where $g_j$ and $k$ are determined by the formula $gg_i=g_j\Tilde{q}_n(\iota(1))^k$. This determines~$g_j$ uniquely and $k$ up to equivalence mod $n$. If we pick~$\theta_n'$ to be the map taking an element in $Q_n$ to the~$g_i$ in the coset corresponding to it we get that $g_j=\theta_n'(\varphi_n'(gg_i))$ and $\tilde{q}_n(\iota(1))^k=gg_i\theta_n'(\varphi_n'(gg_i))^{-1}$. This gives us
$$\pi_n(g)g_i=\omega_n\Big(gg_i(\theta_n'\circ\varphi_n'(gg_i))^{-1}\Big)\cdot\theta_n'\circ\varphi_n'(gg_i).$$
Then we may take $\theta_n$ to be another section from $\Gamma$ to $\Tilde{\Gamma}$ defined so that $\tilde{q}_n\circ\theta_n=\theta_n'\circ q_n$. To show that such a section exists note that for all $g\in\Gamma$ $\tilde{q}_n\circ\theta(g)=\tilde q_n\circ\iota(1)^k\theta_n'\circ q_n(g)$ for some $k\in\Z$. We then define $\theta_n(g)=\theta(g)\iota(1)^{-k}$. We define the $\Gamma$ cocycles $\sigma$, $\sigma_n$ and the $\Gamma$ 1-cochain $\alpha_n$ as follows:
\begin{align*}
\iota(1)^{-\sigma(g,h)}&=\theta(gh)\theta(h)^{-1}\theta(g)^{-1}\\
\iota(1)^{-\sigma_n(g,h)}&=\theta_n(gh)\theta_n(h)^{-1}\theta_n(g)^{-1}\\
\iota(1)^{-\alpha_n(g)}&=\theta_n(g)\theta(g)^{-1}.
\end{align*}
We claim that $\sigma_n$ is of $n$-$Q_n$ type. To show this
\begin{align*}
\tilde{q}_n(\iota(1))^{\sigma_n(g,h)}&=\tilde{q}_n(\theta_n(gh)\theta_n(h)^{-1}\theta_n(g)^{-1})\\
&=\theta_n'(q_n(gh))\theta_n'(q_n(h))^{-1}\theta_n'(q_n(g))^{-1}
\end{align*}
Since the order of $\tilde{q}_n(\iota(1))$ is $n$ it follows that $\sigma_n(g,h)$ reduced mod $n$ can be computed the same way as $q_n^\#$ applied to the cocycle corresponding to $\theta_n'$. Then note that by \cite[Chapter~IV.3 Equation~3.11]{coho} we have that $\sigma_n=\sigma+\partial\alpha_n$. Note that $V_n$ is isomorphic to the vector space for~$\pi_n$ by mapping
$$h\mapsto\theta_n'(h).$$
So we have
$$
\pi_n\circ\Tilde{q}_n\circ\theta(g)\cdot\theta_n'(h)=\omega_n\Big(\Tilde{q}_n\circ\theta(g)\cdot\theta_n'(h)\cdot\big(\theta_n'\circ\varphi_n'(\Tilde{q}_n\circ\theta(g)\cdot\theta_n'(h))\big)^{-1}\Big)\cdot\theta_n'\circ\varphi_n'\Big(\Tilde{q}_n\circ\theta(g)\cdot\theta_n'(h)\Big).$$
To simplify this expression we note first that
\begin{align*}
\varphi_n'\Big(\Tilde{q}_n\circ\theta(g))\cdot\theta_n'(h)\Big)&=\varphi_n'\circ\Tilde{q}_n\circ\theta(g)\cdot\varphi_n'\circ\theta_n'(h)\\
&=q_n\circ\varphi\circ\theta(g)\cdot h\\
&=q_n(g)\cdot h.
\end{align*}
Moreover
\begin{align*}
\Tilde{q}_n\circ\theta(g)&=\Tilde{q}_n(\iota(1)^{\alpha_n(g)}\cdot\theta_n(g))\\
&=\Tilde{q}_n(\iota(1)^{\alpha_n(g)})\cdot\Tilde{q}_n\circ\theta_n(g)\\
&=\Tilde{q}_n(\iota(1)^{\alpha_n(g)})\cdot\theta_n'\circ q_n(g).
\end{align*}
This gives us the formula
$$\pi_n\circ\Tilde{q}_n\circ\theta(g)\cdot\theta_n'(h)=\omega_n\Big(\Tilde{q}_n(\iota(1))^{\alpha_n(g)}\theta_n'(q_n(g))\cdot\theta_n'(h)\cdot\theta_n'(q_n(g)h)^{-1}\Big)\cdot \theta_n'\big(q_n(g)\cdot h\big).$$
Supposing that $h=q_n(a)$ we can rewrite this as
\begin{align*}
\pi_n\circ\Tilde{q}_n\circ\theta(g)\cdot\theta_n'\circ q_n(a)&=\omega_n\Big(\Tilde{q}_n(\iota(1))^{\alpha_n(g)}\theta_n'\circ q_n(g)\cdot\theta_n'\circ q_n(a)\cdot\theta_n'\circ q_n(ga)^{-1}\Big)\cdot\theta_n'\circ q_n(ga)
\\
&=\omega_n\Big(\tilde{q}_n(\iota(1)^{\alpha_n(g)})\tilde{q}_n\big(\theta_n(g)\cdot\theta_n(a)\cdot\theta_n(ga)^{-1}\big)\Big)\cdot\theta_n'\circ q_n(ga)\\
&=\omega_n(\tilde{q}_n(\iota(1)^{\alpha_n(g)}\iota(1)^{\sigma_n(g,a)})\cdot\theta_n'\circ q_n(ga)\\
&=\exp\left(\frac1n(\alpha_n(g)+\sigma_n(g,a))\right)\cdot\theta_n'\circ q_n(ga).
\end{align*}
In the notation of Proposition~\ref{rhon} the $q_n(ga)$ would be written in the form $\bar{g}\bar{a}$ and $\bar{g}\mapsto\theta_n'(g)$ describes the isomorphism of $V_n$ with the vector space on which the induced representation is defined. Thus this recovers the formula from Proposition~\ref{rhon}.
\end{proof}

\section{Examples}

\subsection{$\Z^2$ Revisited}
In this subsection we will apply our results to $\Z^2$, the simplest nontrivial example. We will compare the result of our algorithm to the classical results and show that we get Voiculescu's matrices tensored against another representation. Using remark~\ref{dimReduce} we obtain Voiculescu's matrices precisely. In Example~\ref{z2cos} we have shown that
$$\sigma((x_1,x_2),(y_1,y_2))=x_2y_1$$
is a cocycle on $\Z^2$ and
$$c=[(0,1)|(1,0)]-[(1,0)|(0,1)]$$
is a 2-chain such that
$$\langle\sigma,c\rangle=1.$$
Moreover, since $\sigma$ is a polynomial with integer coefficients it follows that $\sigma$ is $\Z/n\Z$-compatible with $(\Z/n\Z)^2$. Then applying Theorem~\ref{nil} we get that $\rho_n$ acts on $V_n=\ell^2((\Z/n\Z)^2)\cong\C^n\tensor\C^n$. Then pick the basis $\{e_j\tensor e_k\}$ with~$e_j$ defined for all $j\in\Z$ by the formula $e_{j+n}=e_j$. Using the formula for Theorem~\ref{nil} we get that
$$\rho_n(a,b)e_j\tensor e_k=\exp\left(\frac{2\pi i}{n}bj\right)e_{a+j}\tensor e_{b+k}.$$
Note that we may write $\rho_n(a,b)=\rho_n^1(a,b)\tensor\rho_n^2(a,b)$ where
$$\rho_n^1(a,b)e_j=\exp\left(\frac{2\pi i}{n}bj\right)e_{a+j}$$
and
$$\rho_n^2(a,b)e_j=e_{b+j}.$$
Note that $\rho_n^1$ is precisely the asymptotic representation given by Voiculescu's matrices while $\rho_n^2$ is a genuine representation. This is unsurprising because Remark~\ref{dimReduce} allows us to reduce the dimension by ``ignoring'' the second tensor coordinate and the resulting formula is precisely $\rho_n^1$.

\subsection{A 3-Step Nilpotent Group}

Consider the group $\Gamma$ generated by $a_1,\ldots ,a_5$ with the following relations
\begin{align*}
a_2a_1&=a_1a_2a_3\\
a_3a_1&=a_1a_3a_4^{2}\\
a_3a_2&=a_2a_3a_5\\
a_ia_j&=a_ja_i&\mbox{for all other pairs $\{i,j\}$.}
\end{align*}
We first state a simplified version of our formula, then we explain how to compute it. We first compute the general version, then explain how it simplifies in this case. Our asymptotic representation is defined for $n$ co-prime to 6 and sends generators to the following $\C^n$ spanned $e_i$ with $i\in\Z$ and $e_{i+n}=e_i$:
\begin{align*}
\rho_n'(a_1)e_j&=e_{j+1}\\
\rho_n'(a_2)e_j&=\exp\left(\frac{4\pi i}{n}{j\choose 3}\right)e_j\\
\rho_n'(a_3)e_j&=\exp\left(\frac{4\pi i}{n}{j\choose 2}\right)e_j\\
\rho_n'(a_4)e_j&=\exp\left(\frac{2\pi i}{n}j\right)e_j\\
\rho_n'(a_5)&=\Id_{\C^n}.
\end{align*}
And sends the element written uniquely in the form $a_1^{x_1}\cdots a_5^{x_5}\mapsto\rho_n'(a_1)^{x_1}\cdots\rho_n'(a_5)^{x_5}$. Here ${j\choose 3}$ is the polynomial $\frac16j(j-1)(j-2)$. This group can be concretely realized as $\Z^5$ with multiplication given by
$$(x_1,\ldots, x_5)*(y_1,\ldots y_5)=(\eta_1(\mathbf{x,y}),\ldots,\eta_5(\mathbf{x,y}))$$
where
\begin{align*}
\eta_1(\mathbf{x,y})&=x_1+y_1\\
\eta_2(\mathbf{x,y})&=x_2+y_2\\
\eta_3(\mathbf{x,y})&=x_3+y_3+x_2y_1\\
\eta_4(\mathbf{x,y})&=x_4+y_4+2x_3y_1+2x_2{y_1\choose 2}\\
\eta_5(\mathbf{x,y})&=x_5+y_5+y_1{x_2\choose 2}+x_3y_2+x_2y_1y_2
\end{align*}
by the isomorphism
$a_1^{x_1}\cdots a_5^{x_5}\mapsto(x_1,\ldots,x_5)$. In general these polynomials may be calculated by methods given in~\cite{deep}. We will explain how to verify that these polynomials by computer. There is a full description of the code in the appendix, but we will summarize the main steps here.

\begin{enumerate}
    \item Verify that the operation ``$*$'' defined above is associative.
    \item Calling $a_i$ the element with a 1 in the $i$th entry, and zeroes elsewhere, verify that $a_1^{x_1}\cdots a_5^{x_5}=(x_1,\ldots,x_5)$ under the operation $*$.
    \item Use $*$ to compute $a_5^{-x_5}a_4^{-x_4}\cdots a_1^{-x_1}$.
    \item Verify that the formula computed in the previous step is both a left, and a right inverse to $(x_1,\ldots,x_5)$.
    \item Verify that $a_1,\ldots,a_5$ satisfies the relations of the group.
\end{enumerate}

In order to compute a non-torsion cocycle we will develop one as a central extension. We will do this by ``blowing up'' the relation $[a_4,a_1]=e$. Thus we get a group~$\tilde{\Gamma}$ generated by $\tilde{a}_1,\ldots,\tilde{a}_6$ with the relations
\begin{align*}
\tilde a_2\tilde a_1&=\tilde a_1\tilde a_2\tilde a_3\\
\tilde a_3\tilde a_1&=\tilde a_1\tilde a_3\tilde a_4^{2}\\
\tilde a_3\tilde a_2&=\tilde a_2\tilde a_3\tilde a_5\\
\tilde a_4\tilde a_1&=\tilde a_1\tilde a_4\tilde a_6\\
\tilde a_i\tilde a_j&=\tilde a_j\tilde a_i&\mbox{for all other pairs $\{i,j\}$.}
\end{align*}

\begin{rmk}\label{warning}
The reader is warned that such an extension cannot be made for any homogeneous relation in any torsion-free finitely generated nilpotent group. Consider the group $\Lambda$ generated by $b_1,\ldots, b_4$ with the relations
\begin{align*}
[b_1,b_2]&=b_3\\
[b_i,b_j]&=e&\mbox{for all other pairs $\{i,j\}$.}
\end{align*}
In this case the relation $[b_3,b_4]=1$ follows from the other three relations so we cannot construct a central extension by ``blowing it up.'' We will explain why this issue does not arise in our example below. In general an algorithm for finding when relations of this form make a nilpotent group where each ``$a_i$'' has infinite order is described in \cite[Proposition 9.8.3, Proposition 9.9.1]{sims}.
\end{rmk}
Then $\Tilde{\Gamma}$ can be identified with $\Z^6$ with the multiplication is given by 
$$(x_1,\ldots,x_6)*(y_1,\ldots,y_6)=(\gamma_1(\mathbf{x,y}),\ldots,\gamma_6(\mathbf{x,y}))$$
where
\begin{align*}
\gamma_i(\mathbf{x,y})&=\eta_i(x_1,\ldots x_5,y_1,\ldots y_5)&\mbox{for }i<6\\
\gamma_6(\mathbf{x,y})&=x_6+y_6+x_4y_1+2x_3{y_1\choose 2}+2x_2{y_1\choose 3}
\end{align*}
We have verified the fact that these polynomials give rise to a group operation satisfying the relations of the group with similar code to what we used to verify these things for $\Gamma$. Since the element $(0,0,0,0,0,1)$ has infinite order in the group determined by these polynomials it follows that $\tilde a_6$ has infinite order as well.

From this it follows that a cocycle corresponding to the central extension is given by
$$\sigma(\mathbf{x,y})=x_4y_1+2x_3{y_1\choose 2}+2x_2{y_1\choose 3}.$$

Let $c$ be the 2-cycle

$$c=[a_1|a_4]-[a_4|a_1]$$
$$=[(1,0,0,0,0)|(0,0,0,1,0)]-[(0,0,0,1,0)|(1,0,0,0,0)]$$
Then $\langle\sigma,c\rangle=1$.

For any $n$ co-prime to 6 we may define
$$Q_n=(\Z/n\Z)^5$$ with multiplication given by
$$(x_1,\ldots, x_n)*(y_1,\ldots,y_n)=(\bar\eta_1(\mathbf{x,y}),\ldots,\bar\eta_5(\mathbf{x,y}))$$
where $\bar\eta_i$ is $\eta_i$ with each coefficient reduced mod $n$.\footnote{Any $\frac12$ coefficient will be the corresponding inverse of 2 mod $n$ which exists since $n$ is co-prime to 6.} Then reducing each coefficient of $\sigma$ mod $n$ we get $\bar{\sigma}$. Note that the fact that $\bar{\sigma}$ is a $\Z/n\Z$-valued cocycle on $Q_n$ implies that $\sigma$ is of $Q_n$-$n$ type. Then we can use the formula from Remark~\ref{simplify} to define our asymptotic representation. Note that $\ell^2(Q_n)\cong\ell^2(\Z/n\Z)^{\tensor 5}.$ Treat $\{e_i\}_{i=0,\ldots,n-1}$ as a basis for $\ell^2(\Z/n\Z)$. For ease of notation we will extend $e_i$ to be well-defined for all $i\in\Z$ by the formula $e_{i+n}=e_i$. Thus we get
$$\rho_n(x_1,\ldots,x_5)e_{y_1}\tensor e_{y_2}\tensor\cdots\tensor e_{y_5}=\exp\left(\frac{2\pi i\sigma(\mathbf{x,y})}{n}\right)e_{\eta_1(\mathbf{x,y})}\tensor e_{\eta_2(\mathbf{x,y})}\tensor\cdots\tensor e_{\eta_5(\mathbf{x,y})}.$$
Theorem~\ref{coc2} guarantees that this formula is well-defined. As in Remark~\ref{dimReduce} we can ``ignore'' the last tensor coordinate since $\sigma(\mathbf{x,y})$ does not depend on~$y_5$ and neither does $\eta_i$ for $i<5$. This gives us an $n^4$ dimensional asymptotic representation. In this particular case we may go much further. It turns out that $\sigma$ depends only on $\mathbf{x}$ and $y_1$ so we can ignore every tensor coordinate except the first. This gives us the asymptotic representation $\rho_n'$ we introduced in the start of the section.

\subsection{A Polycyclic Group}

Let $\Gamma=\Z^2\rtimes\Z$ where the action of $\Z$ on $\Z^2$ is given by ``Arnold's Cat Map:''
$$1\mapsto T=\begin{bmatrix}2&1\\1&1\end{bmatrix}\in\Gl_2(\Z)=\Aut(\Z^2).$$
The generators are $a_1,a_2,a_3$ with the relations
\begin{align*}
a_2a_1&=a_1a_2\\
a_3a_1&=a_1^2a_2a_3\\
a_3a_2&=a_1a_2a_3.
\end{align*}
A simplified version of our asymptotic representation is given on $\C^n\tensor\C^n$ with basis $\{e_j\tensor e_k\}_{j,k\in\Z}$ and the convention that $e_{j+n}=e_j$. With this notation the generators are sent to the following operators
\begin{align*}
\rho_n'(a_1)e_j\tensor e_k&=e_{j+1}\tensor e_k\\
\rho_n'(a_2)e_j\tensor e_k&=\exp\left(\frac{2\pi i}{n}j\right)e_j\tensor e_{k+1}\\
\rho_n'(a_3)e_j\tensor e_k&=\exp\left(\frac{2\pi i}{n}\left(jk+j^2+\frac12k^2-j-\frac12k\right)\right)e_{2j+k}\tensor e_{j+k}.
\end{align*}

As in the last chapter we will explain how to compute the asymptotic representation given by Proposition~\ref{rhon}, then explain how to derive the simpler formula $\rho_n'$.

We will compute a non-torsion cocycle in $H^2(\Gamma;\Z)$. We explain our reasoning about how to find the cocycle without formal proof then show formally that it obeys the cocycle condition. The idea is as in the previous section to find a central extension of $\Gamma$, compute the multiplication in the middle group of the central extension. We may consider the following presentation of $\Gamma$. 
Each element in the group can be written uniquely as $a_1^{x_1}a_2^{x_2}a_3^{x_3}$ and this element will be sent to the corresponding product of matrices. 

We may make an extension by ``blowing up'' the relation $[a_1,a_2]=1$. Then we may consider $\Tilde{\Gamma}$ to be the group generated by $b_1,\ldots,b_4$ and the relations
\begin{align*}
b_2b_1&=b_4\\
b_3b_1&=b_1^2b_2b_3\\
b_3 b_2&=b_1b_2b_3\\
b_4b_i&=b_ib_4&\mbox{for all $i$.}
\end{align*}
Using these relations we may write any element of $\Tilde{\Gamma}$ uniquely in the form $b_1^{x_1}\cdots b_4^{x_4}$ with $x_i\in\Z$. As the reader is warned in Remark~\ref{warning} it is not always the case that ``blowing up'' a relation like this leads to a sensible extension. In this case when we verify the cocycle condition we will also have verified that this makes a sensible extension. 

We will describe an element of $\Gamma$ implicitly by a pair $(v,k)$ with $v=(v^1,v^2)\in\Z^2$ and $k\in\Z$ then by the definition of the semi-direct product the multiplication is given implicitly by $(v_1,k_1)*(v_2,k_2)=(v_1+T^{k_1}v_2,k_1+k_2)$. Our goal is to implicitly describe $\Tilde{\Gamma}$ similarly. To that end we will describe an element of $\Tilde{\Gamma}$ as a triple $(v,k,d)$ with $v=(v^1,v^2)\in\Z^2$ and $k,d\in\Z$. This represents the element $b_1^{v^1}b_2^{v^2}b_3^{k}b_4^{d}$. To that end we make the following observations:
\begin{align*}
b_2^{v^2}b_1^{v^1}&=b_1^{v^1}b_2^{v^2}b_3^{v^1v^2}&(5.1)\\
b_3a_1^{v^1}&=(b_1^2b_2)^{v^1}b_3\\
&=b_1^{2v^1}b_2^{v^1}b_3b_4^{v^1(v^1-1)}&(5.2)\\
b_3a_2^{v^2}&=(b_1b_2)^{v^2}b_3\\
&=b_1^{v^2}b_2^{v^2}b_3b_4^{\frac12v^2(v^2-1)}&(5.3)\\
b_3b_1^{v^1}b_2^{v^2}&=b_1^{2v^1}b_2^{v^1}b_1^{v^2}b_2^{v^2}b_3a_4^{v^1(v^1-1)+\frac12v^2(v^2-1)}\\
&=b_1^{2v^1+v^2}b_2^{v^1+v^2}b_3a_4^{v^1(v^1-1)+\frac12v^2(v^2-1)}b_4^{v^1v^2}.&(5.4)
\end{align*}
These essentially describe the ways we can get ``$b_4$ terms.'' We will informally refer to the contributions ``(5.1) terms,'' ``(5.2) terms,'' ``(5.3) terms,'' and ``(5.4) terms.'' The ``(5.4) terms'' will refer to the terms in (5.4) that do not appear in (5.2) or (5.3). In order to capture these terms we define the $\alpha,\beta:\Q^2\tensor\Q^2\rightarrow\Q$ and $\gamma:\Q^2\rightarrow\Q$ as follows:
\begin{align*}
\alpha(v_1\tensor v_2)&=v_1^1v_2^2\\
\beta(v_1\tensor v_2)&=\frac12v_1^2v_2^2+v_1^1v_2^1\\
\gamma(v_1)&=v_1^1+\frac12v_1^2.
\end{align*}
Thus we have shown
$$(v_1,1,0)*(v_2,0,0)=(v_1+Tv_2,1,\alpha(v_2\tensor v_1+v_2\otimes v_2)+\beta(v_2\tensor v_2)-\gamma(v_2)).$$
Note that although $\beta$ and $\gamma$ take rational values in general, $\beta(v_2\tensor v_2)-\gamma(v_2)$ is always an integer. Here the input to $\alpha$ comes from the (5.1) and (5.4) terms while the inputs to $\beta$ and $\gamma$ come from the (5.2) and (5.3) terms. To compute the product in general for positive values of $k_1$ the steps would look like
\begin{align*}
(v_1,k_1,0)*(v_2,k_2,0)=&(v_1,k_1-1,0)*(Tv_2,k_2+1,(\alpha+\beta)(v_2\tensor v_2)-\gamma(v_2))\\
=&(v_1,k_1-2,0)*(T^2v_2,k_2+2,\alpha(v_2\tensor v_2+Tv_2\tensor Tv_2)-\gamma(v_2+Tv_2))\\
=&(v_1,k_1-3,0)\\
&*(T^3v_2,k_2+3,\alpha(v_2\tensor v_2+Tv_2\tensor Tv_2+T^2v_2\tensor T^2v_2)\\
&-\gamma(v_2+Tv_2+T^2v_2))
\end{align*}
and so on. This motivates the definition
$$S_k=\begin{cases}\sum_{j=0}^{k-1}(T\tensor T)^j&k\ge 0;\\
-\sum_{j=k+1}^{-1}(T\tensor T)^j&k<0.
\end{cases}$$
We cannot use the exponential sum formula to get a closed form for $S_k$ because $T\tensor T-1$ is not invertible. However $T-1$ is invertible, so we may write a closed form for the analogue of $S_k$ in the linear terms. We have done enough to motivate our definition of the cocycle. It comes from keeping track of each of the ``$b_3$ terms'' when computing multiplication in $\tilde\Gamma$. Note an element of $\Gamma$ as $g_i=(v_i,k_i)=((v_i^1,v_i^2),k_i)$ so that $v_i\in\Z^2$ represent the element $a_1^{v_i^1}a_2^{v_i^2}a_3^{k_i}$. We define
\begin{align*}
\sigma_1((v_1,k_1),(v_2,k_2))&=\alpha(T^{k_1}v_2\tensor v_1)\\
\sigma_2((v_1,k_1),(v_2,k_2))&=\alpha((S_{k_1}(v_2\tensor v_2))\\
\sigma_3((v_1,k_1),(v_2,k_2))&=\beta(S_{k_1}(v_2\tensor v_2))\\
\sigma_4((v_1,k_1),(v_2,k_2))&=\gamma((T-1)^{-1}(T^{k_1}-1)v_2))
\end{align*}
Then we define our cocycle
$$\sigma=\sigma_1+\sigma_2+\sigma_3-\sigma_4.$$
Note that $\sigma_4$ is subtracted unlike the others. Here $\sigma_1$ comes from the ``(5.1) terms,'' $\sigma_2$ comes from the ``(5.4) terms,'' and $\sigma_3-\sigma_4$ comes from the ``(5.2) and (5.3) terms." Before we compute $\partial\sigma$ we observe the following identities about~$S_k$
\begin{align*}
S_{k_1+k_2}&=S_{k_1}+(T\tensor T)^{k_1}S_{k_2}\\
(T\tensor T-1)S_k&=(T\tensor T)^k-1
\end{align*}
Now we compute $\partial \sigma$ piece by piece. We will let $g_i\in\Gamma$ be represented as the pair $(v_i,k_i)$. Then we compute
\begin{align*}
\partial\sigma_1(g_1,g_2,g_3)=&\alpha(-T^{k_1}v_2\tensor v_1+T^{k_1}(v_2+T^{k_2}v_3)\tensor v_1\\
&-T^{k_1+k_2}v_3\tensor(T^{k_1}v_2+v_1)+T^{k_2}v_3\tensor v_2)\\
=&\alpha(-T^{k_1+k_2}v_3\tensor T^{k_1}v_2+T^{k_2}v_3\tensor v_2)\\
=&-\alpha(((T\tensor T)^{k_1}-1)(T^{k_2}v_3\tensor v_2))\\
=&-\alpha((T\tensor T-1)S_{k_1}(T^{k_2}v_3\tensor v_2)).
\end{align*}
Next
\begin{align*}
\partial(\sigma_2+\sigma_3)(g_1,g_2,g_3)=&(\alpha+\beta)(-S_{k_1}(v_2\tensor v_2)+S_{k_1}((v_2+T^{k_2}v_3)\tensor(v_2+T^{k_2}v_3))\\
&-S_{k_1+k_2}(v_3\tensor v_3)+S_{k_2}(v_3\tensor v_3))\\
=&(\alpha+\beta)(-S_{k_1}(v_2\tensor v_2)+S_{k_1}((v_2+T^{k_2}v_3)\tensor(v_2+T^{k_2}v_3))\\
&-S_{k_2}(v_3\tensor v_3)-S_{k_1}(T^{k_2}v_3\tensor T^{k_2}v_3)+S_{k_2}(v_3\tensor v_3))\\
=&(\alpha+\beta)(S_{k_1}(v_2\tensor T^{k_2}v_3+T^{k_2}v_3\tensor v_2)).
\end{align*}
Finally
\begin{align*}
\partial \sigma_4(g_1,g_2,g_3)=&\gamma((T-1)^{-1}(-(T^{k_1}-1)v_2+(T^{k_1}-1)(T^{k_2}v_3+v_2)\\
&-(T^{k_1+k_2}-1)v_3+(T^{k_2}-1)v_3)\\
=&0.
\end{align*}
Let
$$S_{k_1}(T^{k_2}v_3\tensor v_2)=\sum_{i=1}^2u_i\tensor w_i.$$
Note that since $S_{k_1}$ commutes with the map $u\tensor w\mapsto w\tensor u$ by construction so we have that
$$S_{k_1}(v_2\tensor T^{k_2} v_3)=\sum_{i=1}^2w_i\tensor u_i.$$
Now we have
$$\partial\sigma(g_1,g_2,g_3)=\sum_{i=1}^2\Big(\alpha(-(T\tensor T)(u_i\tensor w_i)+2u_i\tensor w_i+w_i\tensor u_i)+\beta(u_i\tensor w_i+w_i\tensor u_i)\Big).$$
Next we see from the definition of $T$ and $\alpha$ that
$$\alpha((T\tensor T)u_i\tensor w_i)=(2u_i^1+u_i^2)(w_i^2+w_i^1)=2u_i^1w_i^1+2u_i^1w_i^2+u_i^2w_i^1+u_i^2w_i^2.$$
Similarly
\begin{align*}
\alpha(2u_i\tensor w_i)&=2u_i^1w_i^2\\
\alpha(w_i\tensor u_i)&=u_i^2 w_i^1\\
\beta(u_i\tensor w_i+w_i\tensor u_i)&=2\beta (u_i\tensor w_i)\\
&=2u_i^1w_i^1+u_i^2w_i^2
\end{align*}
Thus $\partial\sigma=0$. Secondly we have the 2 chain
$$c=[a_2|a_1]-[a_1|a_2]=[((0,1),0)|((1,0),0)]-[((1,0),0)|((0,1),0)].$$
Then since $S_0=0$ and $T^0-1=0$ we have that $\sigma_1$ is the only one of the forms to pair nontrivially with $c$. Thus we see
$$\langle\sigma,c\rangle=1.$$
We next investigate finite quotients of $\Gamma$. If $n,m\in\N^+$ so that the order of $T$ reduced mod $n$ in $\Gl_2(\Z/n\Z)$ divides $m$ then a finite quotient of $\Gamma$ can be of the form
$$(\Z/n\Z)^2\rtimes\Z/m\Z$$
with the action described by the reduction of $T$ mod $n$. Our goal is to find finite quotients $Q_n$ of this form so that $\sigma$ is $n$-$Q_n$ compatible. To do this note that the pair $(S_{k\pm1},T^{k\pm1})$ can be determined from the pair $(S_k,T^k)$ and the entries of $(S_{k\pm1},T^{k\pm1})$ are polynomials in the entries of $(S_k,T^k)$. These polynomials may be reduced mod $n$. So it follows that if we pick $m$ (depending on $n$) so that $S_m\equiv S_0\mod n$ and $T^m\equiv 1\mod n$ we have that
$$(S_k,T^k)\equiv (S_{k+m},T^{k+m})\mod n$$
for all $k$, by induction on $|k|$. It follows that the order of the order of $T$ reduced mod $n$ in $\Gl_2(\Z/n\Z)$. Thus for odd $n$ we define
$$Q_n=(\Z/n\Z)\rtimes\Z/m\Z.$$
Since $n$ is odd we may express the $\frac12$s in the definition of $\sigma$ as the inverse of~2 mod~$n$. We may reduce the rest of the operations mod $n$ easily and the operators~$T^k$ and $S_k$ are $m$-periodic, so the formula for $\sigma$ determines a cocycle~$\sigma'$ which defines a cohomology class in $H^2(Q_n;\Z/n\Z)$. Thus $\sigma$ is of finite type. Our formula for the representation then acts on the space $\C^n\tensor\C^n\tensor\C^m$. We can get a formula from Proposition~\ref{rhon}, but this formula is messy, since our definition for $\sigma$ is messy. For that reason we will simply check what the generators do:
\begin{align*}
\rho_n(a_1)e_j\tensor e_k\tensor e_\ell&=e_{j+1}\tensor e_k\tensor e_\ell\\
\rho_n(a_2)e_j\tensor e_k\tensor e_\ell&=\exp\left(\frac{2\pi i}{n}j\right)e_j\tensor e_{k+1}\tensor e_\ell\\
\rho_n(a_3)e_j\tensor e_k\tensor e_\ell&=\exp\left(\frac{2\pi i}{n}\left(jk+j^2+\frac12k^2-j-\frac12k\right)\right)e_{2j+k}\tensor e_{j+k}\tensor e_{\ell+1}.
\end{align*}
We can see that this must be an asymptotic representation tensored against a genuine representation in the third tensor coordinate. Thus we can pick a smaller representation by ``ignoring'' the $e_\ell$ part. This gives us the list of formula $\rho_n'$ from the start of the subsection.

\appendix

\section{Code to Check Multiplication Polynomials}

We will show sage code that verifies these polynomial give rise to an associative relation that obeys the given presentation of the group, by the isomorphism described above. We first enter the polynomials:
\begin{center}
\includegraphics[scale=.85]{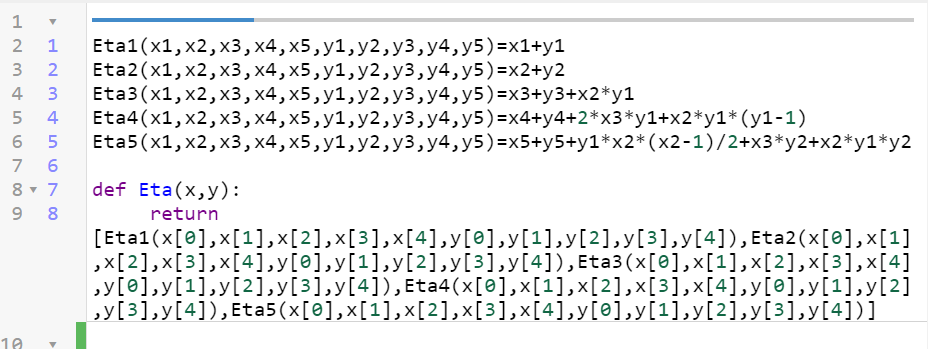}
\end{center}

Here the function ``Eta'' should take in 2 lists of 5 numbers or algebraic expressions (representing two elements of the group) and apply the 5 polynomials to them, outputting another list of 5 elements (representing the product of those two elements). Next we check associativity:
\begin{center}
\includegraphics[scale=.85]{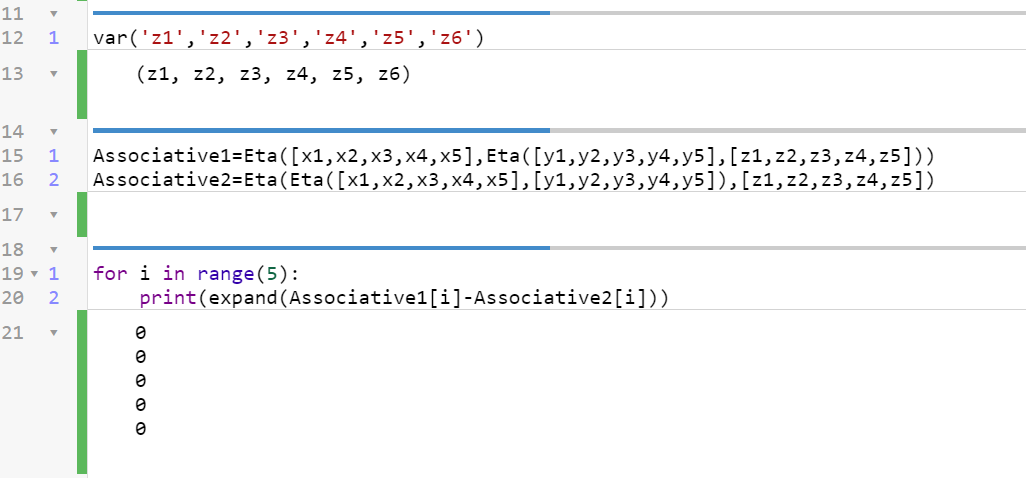}
\end{center}
The first line of code is needed to make sage treat $z_1,\ldots,z_6$ as algebraic expressions. The next two lines compute $\mathbf{x*(y*z})$ and $\mathbf{(x*y)*z}$ respectively. The for loop checks that these expressions are equivalent in each coordinate. It is easy to see by inspecting the polynomials that $(1,0,0,0,0)^{x_1}=(x_1,0,0,0,0)$, $(0,1,0,0,0)^{x_2}=(0,x_2,0,0,0)$, and so on. The following code, assumes this fact and checks that the relation $(x_1,\ldots,x_5)=a_1^{x_1}\cdots a_5^{x_5}$ makes sense with $a_i$ corresponding to the vector that has a 1 in the $i$th place and zeroes elsewhere:
\begin{center}
\includegraphics[scale=.85]{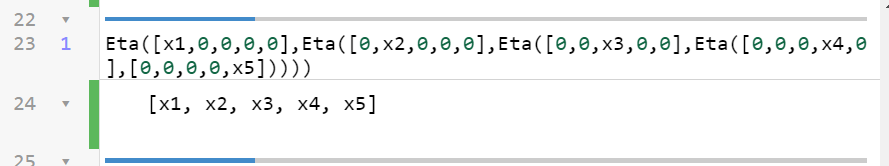}
\end{center}
Next we must check the existence of inverses:
\begin{center}
\includegraphics[scale=.85]{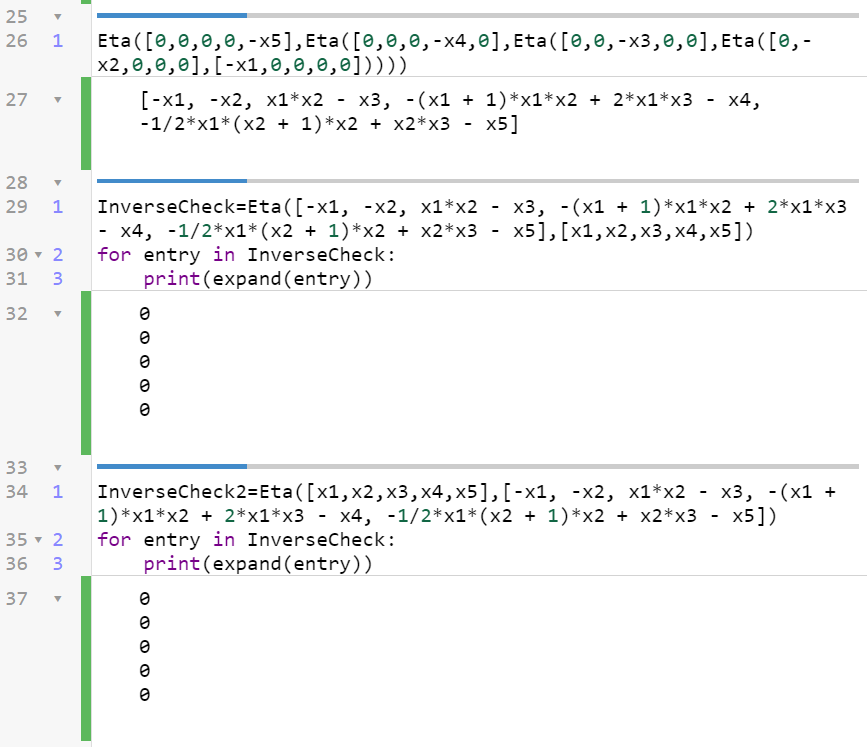}
\end{center}
The first line of code computes what the inverse of $(x_1,\ldots,x_5)$ must be if it exists. The next two lines verify that this is in fact, both a left and right inverse, respectively. Finally, we need to check that these relations satisfy the the presentation of the group. The first few of the computations look like this

\begin{center}
\includegraphics[scale=.85]{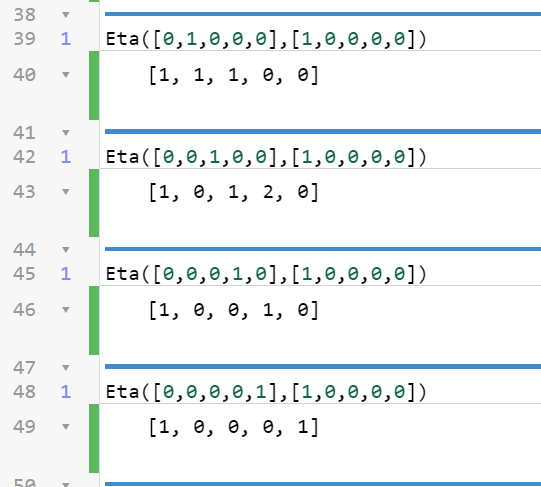}
\end{center}

These lines verify the relations $a_2a_1=a_1a_2a_3$, $a_3a_1=a_1a_3a_4^2$, and $a_5a_1=a_1a_5$, respectively. The rest of the relations may be checked with similar code.

\begin{center}
\textbf{Acknowledgments}
\end{center}
I would like to thank my advisor Marius Dadarlat for introducing me to the problem and for all his other advice. I would like to thank Ben McReynolds and Mark Pengitore for answering questions I had. I would also like to thank the Purdue University mathematics department for supporting me in the summer of 2021 with the summer research grant. Finally, I would like to thank an anonymous referee for their hard work and many helpful comments.

\medskip
\bibliographystyle{plain}
\bibliography{main}
\end{document}